\documentclass{amsart}
\usepackage{amsmath,amssymb,graphicx,setspace,verbatim,url}
\usepackage{color}
\usepackage{hyperref}
\usepackage{multicol}
\usepackage{enumitem}
\usepackage{float}
\usepackage{tikz}
\usepackage{soul}

\input xy
\xyoption{all}
\newtheorem{prop}{Proposition}[section]
\newtheorem{coro}[prop]{Corollary}
\newtheorem{thm}[prop]{Theorem}
\newtheorem{lemma}[prop]{Lemma}

\newcommand{\PSL}{\mathrm{PSL}}

\begin{document}

\title{String C-group representations of transitive Groups: a case study with degree $11$}

\author[M. E. Fernandes]{Maria Elisa Fernandes}
\address{Maria Elisa Fernandes,
Center for Research and Development in Mathematics and Applications, Department of Mathematics, University of Aveiro, Portugal
}
\email{maria.elisa@ua.pt}

\author[Claudio Piedade]{Claudio Alexandre Piedade}
\address{Claudio Alexandre Piedade, Centro de Matemática da Universidade do Porto
}
\email{claudio.piedade@fc.up.pt}

\author[Olivia Reade]{Olivia Reade}
\address{Olivia Reade,
Open University, Milton Keynes, MK7 6AA, U.K.
}
\email{olivia.reade@open.ac.uk}

\date{}
\maketitle


\begin{abstract}
In this paper we give a non-computer-assisted proof  of the following result: if $G$ is an even transitive group of degree $11$ and has a  string C-group representation with rank $r\in\{4,5\}$ then $G\cong\PSL_2(11)$. Moreover this string C-group is 
the group of automorphisms of the rank $4$ polytope known as the $11$-cell. 

The insights gained from this case study include techniques and observations concerning permutation representation graphs of string C-groups. The foundational lemmas yield a natural and intuitive understanding of these groups. These and similar approaches can be replicated and are applicable to the study of other transitive groups.

\end{abstract}


\noindent \textbf{Keywords:} Abstract Regular Polytopes; String C-Groups; Symmetric Groups; Alternating Groups; Permutation Groups.

\noindent \textbf{2000 Math Subj. Class:} 52B11, 20D06.


\section{Introduction}\label{intrud}

 It is well-known that abstract regular polytopes are in one to one correspondence with string C-groups \cite{arp}. In this day and age, and given the right circumstances in terms of access to sufficiently powerful computing technology, it is possible to create, by computer, classifications of abstract regular polytopes for any given rank and ``small-enough'' group. In contrast, this paper presents a detailed exposition of a variety of computer-free methods by which one may approach such a problem. The example on which we focus is even  permutation groups of degree $11$, and in this way we obtain a classification of such abstract regular polytopes for rank 4 or 5. This provides an illustrative demonstration of a methodology for classifying string C-groups, and it establishes a pathway for tackling unsolved open-problems such as the classification of high-rank string C-groups for alternating groups of arbitrary degree.

The ``Aveiro theorem'' states that the maximal rank of an abstract regular polytope with alternating group of degree $n$ as its automorphism group is $\lfloor \frac{n-1}{2}\rfloor$ when $n\geq 12$ \cite{2017CFLM}. 
For the  alternating groups of degrees $5$, $9$, $10$ and $11$,  the maximal ranks are $3$, $4$, $5$ and $6$, respectively.  The remaining alternating groups are not the automorphism groups of any such polytope. In \cite{2019fl} it was proved that there exists an abstract regular polytope for each rank 
$r\in\{3,\ldots, \lfloor \frac{n-1}{2}\rfloor\}$ when $n\geq 12$. The set of all possible ranks of abstract regular polytopes for alternating groups with a degree different from $11$ is either empty or an interval, as shown in Table~\ref{Rankset}.

\begin{small}
\begin{table}
\begin{center}
\begin{tabular}{||c|c||}
\hline
Group&Set of ranks\\[2pt]
\hline
$A_5$&\{3\}\\[5pt]
$A_6$&$\emptyset$\\[5pt]
$A_7$&$\emptyset$\\[5pt]
$A_8$&$\emptyset$\\[5pt]
$A_9$&\{3,4\}\\[5pt]
$A_{10}$&\{3,4,5\}\\[5pt]
$A_{11}$&\{3,6\}\\[5pt]
$A_n$, $n\geq 12$&$ \{3,\ldots,  \lfloor (n-1)/2 \rfloor\}$\\[5pt]
\hline
\end{tabular}
\end{center}
\caption{The set of possible ranks of abstract regular polytopes for  $A_n$ for each $n\geq 5$}
\label{Rankset}
\end{table}
\end{small}
The alternating group $A_5$ is the first alternating group that is the group of automorphisms  of a regular polytope, namely there are, up to duality, exactly two abstract regular polytopes for $A_5$, the hemi-icosahedron and the hemi-great dodecahedron.
In his doctoral thesis Conder proved that all but finitely many alternating groups are the automorphism group of a regular map of type $\{3,m\}$ with $m>6$ (this result can also be found in \cite{1980MC,1981MC}). 
As regular maps for alternating groups are precisely  abstract regular polyhedra \cite[Corollary 4.2]{CO},  this means that the number $3$ belongs to each set of ranks of Table~\ref{Rankset},  the exceptions being $n=3,\,4,\,6,\,7$ or $8$.
The lists of all abstract regular polytopes for alternating groups up to degree $9$ are available in \cite{LVatlas}. In \cite{flm} the authors give  permutation representation graphs of all abstract regular polytopes for $A_9$ and $A_{10}$ having ranks $r\in\{4,5\}$,  and some examples of rank $6$ abstract regular polytopes for the group $A_{11}$. In  \cite{flm2}, their computations revealed the non-existence of abstract regular polytopes of ranks $4$ and $5$ for $A_{11}$.  
In 2018, Meynaert in his master's thesis \cite{MT} gave a complete classification of the representations of $A_{11}$ as a string group generated by an independent set of involutions with rank $4$ or $5$. 
Meynaert used permutation representation graphs in his classification approach, but he did not explore the potential of fracture graphs in his work.

During a problem session in the 2022 Edition of the Symmetries in Graphs, Maps, and Polytopes Workshop, $A_{11}$ was again highlighted as an interesting case of study because it is the unique known example of a group whose set of ranks is not an interval. 

 The $11$-cell is a rank $4$ polytope discovered by Coxeter and Gr\"{u}nbaum in the 80's, and is the only known abstract regular polytope with rank $r\in\{4,5\}$ having an even transitive group of degree $11$ as its automorphism group, namely the $\PSL_2(11)$. 
The group $\PSL_2(11)$ is the unique transitive even group of degree $11$ which is the automorphism group of an abstract
regular polytope having rank $4$ or $5$. Moreover the only rank $4$ polytope for  $\PSL_2(11)$ is the $11$-cell, which is self-dual, and there is no abstract regular polytopes of  rank $5$ for $\PSL_2(11)$.

Our approach to show this result uses the concept of fracture graphs, as first introduced in \cite{2023CFL}, which provides a method for tackling the problem, dividing it into three distinct cases: absence of a fracture graph, presence of a split, and the existence of a 2-fracture graph. This method gives a way to determine string group generated by involutions representations of a transitive group \cite{2023CFL,extension}, such as the illustrative example of even groups of degree $11$.
 A string group generated by involutions may not be a string C-group, then it is necessary to test whether the intersection property is satisfied.
For groups of degree $11$, this evaluation is straightforward using computer-based methods. However, in contrast to a simple ``yes'' or ``no''  outcome, our approach provides a more profound understanding by elucidating the reasons behind the failure.

In the first four sections we give the tools that will be used in this classification but that can also be used in a more general setting.
\begin{itemize}
\item Section~\ref{2}: String C-groups.
\item Section~\ref{3}: Permutation representation graphs.
\item Section~\ref{4}: Fracture graphs.
\item Section~\ref{5}: Conditions leading to the failure of the intersection property.
\end{itemize}

In the following remaining sections, we show how the tools described above can be used on our example, in which we assume that $\mathcal{G}$ is any permutation representation graph of an even transitive string C-group of degree $11$. 
We start by dividing into the cases where $\mathcal{G}$ has a fracture graph (with a split or without a split), and after we give a classification that shows what we have claimed above. 
\begin{itemize}
\item Section~\ref{even11split}: When $\mathcal{G}$ has a fracture graph with a split.
\item Section~\ref{2fracture11}: When $\mathcal{G}$ has a $2$-fracture graph.
\item Section~\ref{final}: A classification of even transitive string C-groups of degree $11$.
\end{itemize}

Our results rely on the atlas of finite groups and on classifications of regular polyhedra for $\PSL_2(11)$ which are available and well known among the researchers working on abstract polytopes and maps. 

\section{String C-groups}~\label{2}

A group $G$ is the automorphism group of an abstract regular polytope of rank $r$ if and only if it has a \emph{string C-group representation}  $\Gamma=(G, \{\rho_0,\ldots,\rho_{r-1}\})$ such that:
\begin{enumerate}
\item[(1)] $G=\langle \rho_0,\ldots, \rho_{r-1}\rangle$;
\item[(2)]  $\{\rho_0,\ldots,\rho_{r-1}\}$ is an ordered set of involutions;
\item[(3)] $\forall i,j\in\{0,\ldots, r-1\}, \;|i-j|>1\Rightarrow (\rho_i\rho_j)^2=1$ (commuting property);
\item[(4)] $\forall J, K \subseteq \{0,\ldots,r-1\}, \langle \rho_j \mid j \in J\rangle \cap \langle \rho_k \mid k \in K\rangle = \langle \rho_j \mid j \in J\cap K\rangle$.
\end{enumerate}
The sequence $\{p_1,\,\ldots,\,p_{r-1}\}$ where $p_i$ is the order of $\rho_{i-1}\rho_i$ is  the  \emph{(Schl\"afli) type} of $\Gamma$. 
A representation $\Gamma=(G, \{\rho_0,\ldots,\rho_{r-1}\})$  that satisfies (1), (2) and (3) is called a \emph{string group generated by involutions} or, for short, a sggi. The \emph{dual} of a sggi is obtained by reversing the sequence of generators.

Let us consider the following notation.
\begin{eqnarray*}
\Gamma_{i_1,\ldots,i_k} &:=&(G_{i_1,\ldots,i_k}, \{\rho_j : j \notin \{i_1,\ldots,i_k\}\});\\
\Gamma_{\{i_1,\ldots,i_k\}} &:=&(G_{\{i_1,\ldots,i_k\}}, \{\rho_j : j \in \{i_1,\ldots,i_k\}\});\\
\Gamma_{<i} &:=& (G_{<i},\{ \rho_0,\ldots, \rho_{i-1}\})\qquad (i\neq 0);\\
\Gamma_{>i} &:=& (G_{>i},\{ \rho_{i+1},\ldots, \rho_{r-1}\})\qquad (i\neq r-1).
\end{eqnarray*}
The {\em maximal parabolic subgroups} of $\Gamma$ are the subgroups $G_i$ with $i\in \{0,\ldots,r-1\}$.

The following result shows that when $\Gamma_0$ and $\Gamma_{r-1}$ are string C-groups, the intersection property for $\Gamma$ is verified by checking only one condition.

\begin{prop}\cite[Proposition 2E16]{arp}\label{arp1}
Let $\Gamma=(G, S)$ be a sggi with $S:=\{\rho_0,\ldots,\rho_{r-1}\}$.
Suppose that $\Gamma_0$ and $\Gamma_{r-1}$ are string C-groups. If $G_0\cap G_{r-1} \cong G_{0,r-1}$, then $\Gamma$ is a string C-group.
\end{prop}


\subsection{Sesqui-extensions}

The term sesqui-extension was first introduced in \cite{flm}. Let us recall its meaning.
Let $\Gamma=(G, \{\rho_0,\ldots,\rho_{r-1}\})$ be a sggi, and let $\tau$ be an involution in a supergroup of $G$ such that $\tau \not \in G$ and $\tau$ centralizes $G$.  
For a fixed $k$, we define the sggi $\Gamma^*=(G^*, \{ \rho_i \tau^{\eta_i}\mid i\in \{0,\,\ldots,\,r-1\} )$ where $\eta_i = 1$ if $i=k$ and $0$ otherwise, the {\it sesqui-extension} of $\Gamma $ with respect to $\rho_k$ and $\tau$.

\begin{prop} \label{sesqui0}\cite[Proposition 3.3]{flm}
If $\Gamma$ is a string C-group, and $\Gamma^*$ is a sesqui-extension of $\Gamma$ with respect to the first generator, then
$\Gamma^*$ is a string C-group.
\end{prop}

\begin{lemma}\cite[Lemma 5.4]{flm2}\label{sesqui} 
Let $\Gamma=(G, \{\rho_0,\ldots,\rho_{r-1}\})$ be a sggi. If $\Gamma^*=(G^*, \{ \rho_i \tau^{\eta_i}\mid i\in \{0,\,\ldots,\,r-1\} )$ where $\eta_i = 1$ if $i=k$ and $0$ otherwise, then the following hold:
\begin{enumerate}
\item $G^*\cong G$ or $G^*= G\times \langle \tau\rangle\cong G\times C_2$. 
\item If the identity element of $G$ can be written as a product of generators involving $\rho_k$ an odd number of times, then $G^*= G\times \langle \tau\rangle$.
\item If $G$ is a finite permutation group, $\tau$ and $\rho_k$ are odd permutations, and all other $\rho_i$ are even permutations, then $G^*\cong G$.
\item Whenever $\tau\notin  G^*$, $\Gamma$ is a string C-group if and only if $\Gamma^*$ is a string C-group.
\end{enumerate}
\end{lemma}

\section{Permutation representation graphs}~\label{3}

Suppose that $G$ is a permutation group of degree $n$ and let $\Gamma=(G, \{\rho_0,\ldots,\rho_{r-1}\})$ be a sggi. The \emph{permutation representation graph} $\mathcal{G}$ of $\Gamma$ is an $r$-edge-labelled multigraph with $n$ vertices and with an $i$-edge $\{a,\,b\}$ whenever $a\rho_i=b$ with $a\neq b$.  
The \emph{dual} of a permutation representation graph is obtained by reverting the labels of the edges according to the correspondence $(0,\ldots,r-1)\mapsto  (r-1,\ldots,0)$. 
Let $\mathcal{G}_{i_1,\ldots,i_k}$ (resp. $\mathcal{G}_{\{i_1,\ldots,i_k}\}$) denote the permutation representation graph of  $\Gamma_{i_1,\ldots,i_k}$ (resp. $\Gamma_{\{i_1,\ldots,i_k}\}$). Notice that when $\rho_i$ is a $k$-transposition (a  product of $k$ disjoint transpositions), $\mathcal{G}_{\{i\}}$ is a matching with $k$ edges.

If  $a\rho_i = a\rho_j = b$ with $a \neq b$ and $i\neq j$ then we say that the graph has a \emph{double $\{i,j\}$-edge}. Similarly, triple edges with labels $i$, $j$ and $k$ are called \emph{triple $\{i,j,k\}$-edges}. These multiple edges are represented as follows (respectively). 
$$\xymatrix@-1.6pc{*+[o][F]{a}  \ar@{=}[rr]^{\{i,j\}} && *+[o][F]{b}}\qquad \xymatrix@-1.6pc{*+[o][F]{a}  \ar@{-}[rr] \ar@{=}[rr]^{\{i,j,k\}} && *+[o][F]{b}} $$
A square with alternating labels in the set $\{i,j\}$ is called an \emph{$\{i,j\}$-square}.

A consequence of the commuting property is that, if $i$ and $j$ are nonconsecutive the connected components of $\mathcal{G}_{\{i,j\}}$ with more then two vertices are $\{i,j\}$-squares.  We also have the following lemma which is a direct consequence of the commuting property.

\begin{lemma}\label{Olivia1}
If $j$ is the label of an edge of $\mathcal{G}$ connecting a vertex of $\mathrm{ Fix}(\rho_i)$ and a vertex of its complement $\overline{\mathrm{Fix}(\rho_i)}$ then $j\in\{i-1,i+1\}$. 
\end{lemma}

Another consequence of the commuting property is that $\rho_0$ centralizes $G_{0,1}$, for that reason we may state the following results about the connected components of $\mathcal{G}_{0,1}$. The dual of the following lemmas also can be applied to $\mathcal{G}_{r-1,r-2}$.

\begin{lemma}\label{l0}
Let $U$ and $V$ be distinct $G_{0,1}$-orbits.
\begin{enumerate}
\item If $x\rho_0=y$ with $x,y\in U$ and $x\neq y$, then $|U|$ is even.
\item If $x\rho_0=y$ with $x\in U$ and $y\in V$, then the permutation representation subgraph of $\mathcal{G}_{0,1}$ induced by $U$ is a copy of  the one induced by $V$.
\end{enumerate}
\end{lemma}
\begin{proof}
This is a consequence of the commuting property of $\Gamma$.
\end{proof}

\begin{lemma}\label{l1}
If $G_0$ is transitive and $\rho_0$ is an even permutation then one of the following situations occurs.
\begin{enumerate}
\item $G_{0,1}$ has at least one orbit of even size.
\item $G_{0,1}$ has at least four odd orbits.
\end{enumerate}
\end{lemma}
\begin{proof}
Suppose that all $G_{0,1}$-orbits are odd. Then, by Lemma~\ref{l0}, $\rho_0$ cannot swap a pair of vertices in the same $G_{0,1}$-orbit.
Then $\rho_0$ swaps vertices in different $G_{0,1}$-orbits pair-wisely. 
Let $O_1$ and $O_2$ be $G_{0,1}$-orbits such that  $O_1\rho_0=O_2$. If $\rho_0$ fixes the remaining points, then  $\rho_0$ is a product of $|O_1|(=|O_2|)$ transpositions, hence $\rho_0$ is odd, a contradiction.
Thus there exists another pair of (odd) orbits $O_3$ and $O_4$ such that $O_3\rho_0=O_4$, as wanted.
\end{proof}

\begin{lemma}\label{l2}
If the permutation representation subgraphs induced by each of the $G_{0,1}$-orbits are all different, then $\rho_0$ acts non-trivially  only on $G_{0,1}$-orbits of even size (fixing the odd orbits pointwisely). 

\end{lemma}
\begin{proof}
This is an immediate consequence of Lemma~\ref{l0} (b).
\end{proof}

\section{Fracture graphs}~\label{4}

Suppose that all maximal parabolic subgroups of $\Gamma$ are intransitive. A \emph{fracture graph} of $\mathcal{G}$ is a subgraph of $\mathcal{G}$ having $n$ vertices and, for each $i\in\{0,\ldots, r-1\}$, one $i$-edge chosen among the $i$-edges between vertices in different $G_i$-orbits \cite{extension}.
A fracture graph of $\mathcal{G}$ thus has exactly $r$ edges.   
 
In general a sggi has multiple fracture graphs. Indeed only $S_n$ has a string C-group representation, corresponding to the simplex, having a uniquely determined fracture graph.
An $i$-edge that belongs to every fracture graph of $\mathcal{G}$ is called an \emph{$i$-split} of $\Gamma$ \cite{2023CFL}.  A split is a bridge of $\mathcal{G}$, therefore it satisties the following property.

\begin{prop}\label{path}
Any path (not containing an $i$-edge) from an $i$-split to an edge with label $l$, where $l \neq i$, contains all labels between $l$ and $i$. 
\end{prop}
\begin{proof}
This is a consequence of Proposition 5.18 of \cite{2017CFLM}.
\end{proof}

\begin{lemma}\label{claudio}
Let $\Gamma:=(G, \{ \rho_0, \ldots, \rho_{n-1}\})$ be a sggi with a permutation representation graph $\mathcal{G}$ having  a fracture graph. 
If $\rho_i$ is a $2$-transposition  and $\mathcal{G}$ has a double $\{i,j\}$-edge,  for some $i,j\in\{0,\ldots, n-1\}$, then $\mathcal{G}$ has an $i$-split.
\end{lemma}
\begin{proof}
This is an immediate consequence of a definition of a split.
\end{proof}

Suppose that $\mathcal{G}$ admits a fracture graph.
If in addition $\mathcal{G}$ has no splits then, for every $i\in\{0,\ldots, r-1\}$, there are at least two $i$-edges between vertices in different $G_i$-orbits. In this case $\mathcal{G}$ admits a \emph{2-fracture graph}, that is a subgraph of $\mathcal{G}$ with $n$ vertices and with exactly two $i$-edges between vertices in different $G_i$-orbits, for each $i\in\{0,\ldots, r-1\}$ \cite{2017CFLM}. A 2-fracture graph of $\mathcal{G}$  thus has exactly $2r$ edges.

\section{Conditions leading to the failure of the intersection property}\label{5}

 In this section, we give sufficient conditions for the intersection property to fail. 
{ \st Relying on these, we now prove that all sggi's given in the appendix are not string C-groups. These sggi's are a result of the case-by-case analysis.}

\begin{prop}\cite[Proposition 6.1]{2023CFL}\label{SNC} 
\begin{enumerate}
\item[(I)] Let $G$ be a primitive permutation group containing a $3$-cycle.
Then $G$ is the alternating or symmetric group.
\item[(II)] Let $G$ be an intransitive permutation group containing a 
$3$-cycle $\alpha$. Let $X$ be the orbit of one of the points of $\alpha$,
and $H$ the group induced on $X$ by $G$. If $A_X\le H$, then $A_X\le G$.
\end{enumerate}
\end{prop}



\begin{lemma}\label{IPF4}
Let $\Gamma=(G,\{\rho_0,\rho_1,\rho_2,\rho_3\})$ be an even sggi. Suppose that
\begin{itemize}
\item $D$ is a $G_{0,3}$-orbit with at least four points,
\item  $X$ is the $G_0$-orbit containing $D$ and
\item $Y$ is the $G_3$-orbit containing $D$.
\end{itemize}
If the following two conditions hold then $\Gamma$ is not a string C-group.
\begin{enumerate}
\item $G_0$ is primitive on $X$ and there exists a permutation $\alpha\in G_0$ such that $\alpha$ is a $3$-cycle on $X$, fixing the complement $\bar{X}$ point-wisely;
\item $G_3$ is primitive on $Y$ and there exists a permutation $\beta\in G_3$ such that $\beta$ is a $3$-cycle on $Y$, fixing the complement $\bar{Y}$ point-wisely. \end{enumerate}
\end{lemma}
\begin{proof}
Suppose that the conditions (a) and (b) are satisfied. By Proposition~\ref{SNC} we conclude that $G_0$  contains all even permutations on $X$, that is, $A_{X}\leq  G_0$. In particular $A_D\leq  G_0$.
Similarly we get that $A_D\leq  G_3$.
Hence $A_D\leq G_0\cap G_3$ and $|D| \geq 4$.
But $G_{0,3}$ is a dihedral group, thus $G_0\cap G_3 \neq G_{0,3}$.
\end{proof}

Using Lemma~\ref{IPF4}, the failure of the intersection property of the permutation representations given in the appendix can be proven for most cases. For the remaining rank $4$ cases, the proof of the following proposition also gives an alternative approach, relying heavily on the fact that $G_{0,3}$ is a dihedral group with an intransitive action. This allows us to find a permutation in $G_0\cap G_3$ that does not belong to that dihedral group.

\begin{prop}\label{notC}
The sggi's of the appendix are independent generating sets for $A_{11}$  but they are not string C-groups.
\end{prop} 
\begin{proof} In all cases of the appendix, we have that $G$ is a transitive permutation group of prime degree $11$, hence primitive.
Suppose that $\Gamma=(G,\{\rho_0,\rho_1,\rho_2,\rho_3\})$ is the sggi corresponding to the graph (A1) with the following numeration of the vertices.
$$ \xymatrix@-1.7pc{*+[o][F]{1}  \ar@{-}[rr]^1&& *+[o][F]{2} \ar@{-}[rr]^0&&*+[o][F]{3} \ar@{-}[rr]^1&& *+[o][F]{4} \ar@{=}[rr]^{\{0,2\}}&& *+[o][F]{5} \ar@{-}[rr]^1&& *+[o][F]{6}  \ar@{-}[rr]^2&&  *+[o][F]{7}  \ar@{-}[rr]^3&& *+[o][F]{8} \ar@{-}[rr]^2&&*+[o][F]{9} \ar@{=}[rr]^{\{1,3\}}&&*+[o][F]{10} \ar@{-}[rr]^2&&*+[o][F]{11} }$$
Consider the permutations $a\in G_0$ and $b\in G_3$ defined as follows. 
\[\begin{array}{l}
a:=\rho_1\rho_2\rho_3:=(1,2)(3,5,8,10,7,6,4)(9,11)\\
b:=\rho_0\rho_1\rho_2=(1,2,5,3)(4,7,6)(8,9,11,10)
\end{array}\]
Since the generators of $G$ are even and $b^4$ is a $3$-cycle, by Proposition~\ref{SNC}, we conclude that $G\cong A_{11}$.
Now consider the sets:
\[D:=\{3,\ldots,7\}, \,X=\{3,\ldots,11\}\mbox{ and }Y:=\{1,\ldots,7\}.\]

Consider the action of $G_0$ on $X$. As $|X|=9$ and $a$ has a $7$-cycle on its cyclic decomposition permuting  elements of $X$, $G_0$ is primitive  on $X$. As $|Y|=7$, $G_3$ is primitive on $Y$. Notice that $a^7$ and its conjugate by $\rho_1$ are both 2-transpositions whose product is a $3$-cycle, that is $\alpha:=a^7(a^7)^{\rho_1}= (9,10,11)$ is a $3$-cycle, satisfying condition (a) of Lemma~\ref{IPF4}. 
Finally the permutation $\beta:=b^4$ satisfies the (b) of Lemma~\ref{IPF4}, thus $\Gamma$ is not a string C-group.

The remaining sggi of the appendix can be dealt in the same way, with few exceptions. Namely for graphs (B14), (B15), (C1), (D1), (D2) and (F3) a different argument should be applied. 
In these cases $I:=G_0\cap G_3$ acts on each orbit as a dihedral group but it happens that $I$ is a bigger group. Let $\Gamma=(G,\{\rho_0,\rho_1,\rho_2,\rho_3\})$  be the sggi corresponding to the graph
(B14) of the appendix  with the following numeration of the vertices.
$$
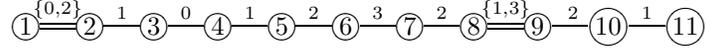
\xymatrix@-1.7pc{*+[o][F]{1}  \ar@{=}[rr]^{\{0,2\}} && *+[o][F]{2} \ar@{-}[rr]^1&&*+[o][F]{3} \ar@{-}[rr]^0&& *+[o][F]{4} \ar@{-}[rr]^1&&*+[o][F]{5} \ar@{-}[rr]^2&& *+[o][F]{6} \ar@{-}[rr]^3&& *+[o][F]{7} \ar@{-}[rr]^2&&*+[o][F]{8} \ar@{=}[rr]^{\{1,3\}}&&*+[o][F]{9} \ar@{-}[rr]^2&& *+[o][F]{10} \ar@{-}[rr]^1&&*+[o][F]{11} }$$
By similar arguments as before we conclude that $G_0$ acts as a symmetric group on the set $\{4,\ldots,11\}$. In particular, as $G$ is even, either  $(7,8)(9,10)\in G_0$  or $(7,8)(9,10)(1,2,3)\in G_0$. 
Since $((7,8)(9,10)(1,2,3))^3=(7,8)(9,10)$ in either case we have that  $(7,8)(9,10)\in G_0$. In addition $(\rho_2(\rho_1\rho_0)^2)^3=(7,8)(9,10)\in G_3$.
This implies that $(7,8)(9,10)\in G_0\cap G_3$. Moreover, $\rho_2(7,8)(9,10)=(1,2)(5,6)\in G_0\cap G_3$, as well as its conjugate by $\rho_1$. But then $S_3\times D_{10}\leq G_0\cap G_3$, and therefore $\Gamma$ is not a string C-group.
Similar arguments can be used when $\Gamma$ has one of the permutation representation graphs (B15), (C1), (D1), (D2) or (F3).

For the sggis of rank $5$ in the appendix, it can be shown that $\Gamma_0$ is not a  string C-group, consequently $\Gamma$ is not a string  C-group. 
\end{proof}


\section{When $\mathcal{G}$ has a fracture graph with a split}\label{even11split}

Let $G$ be  an even transitive group of degree $11$ and rank $r\in\{3,4,5\}$.  Notice that $r\neq 3$ otherwise  $\mathcal{G}_i$ has nine edges with precisely two labels demanding that one of the permutations is odd. In addition assume that $G_j$ is intransitive for every $j\in\{0,\ldots,r-1\}$. 
Suppose that $\mathcal{G}$ has a split $\{a,b\}$ with label $i$.
Then $G_i$ has exactly two orbits $O_1$ and $O_2$. Let $a\in O_1$ and $b\in O_2$.
For $j\neq i$, $\rho_j=\alpha_j\beta_j$ where $\alpha_j$ acts on $O_1$ and $\beta_j$ acts on $O_2$, and $\rho_i=\alpha_i\beta_i(a,b)$ where $\alpha_i$ acts on $O_1$ and $\beta_i$ acts on $O_2$.
Let $J_A := \{j \in \{0, \ldots, r-1\}\setminus\{i\}\mid\alpha_j \neq 1_G\}$ and $J_B := \{j \in \{0, \ldots, r-1\}\setminus\{i\}\mid \beta_j \neq 1_G\}$. We then have $A=\langle \alpha_j\mid j\in J_A\rangle$ and $B=\langle \beta_j\mid j\in J_B\rangle $. If one of the groups is trivial then the corresponding set of indices is empty.

 In what follows $m_i$ and $k_i$ denote, respectively,  the number of blocks and  the size of a block for an imprimitive action on $O_i$.

\begin{prop}\cite[Proposition 5.1]{2017CFLM}\label{CCD}
If $B$ is primitive, then the set $J_B$  is an interval. The same result holds for $A$. 
\end{prop}

 We start by considering that the $\mathcal{G}_i$ components have at least two vertices. Later we deal with the other case, where one of the components is trivial.

\subsection{Case: $\mathcal{G}_i$ has two nontrivial components.}

\begin{prop}\label{B1}
$A$ and $B$ cannot both be imprimitive.
\end{prop}
\begin{proof}
Suppose that $A\leq S_{k_1}\wr S_{m_1}$ and  $B\leq S_{k_2}\wr S_{m_2}$ with $k_1,\, k_2,\, m_1,\,m_2\geq 2$.
Then $k_1m_1+k_2m_2=11$.
But then either  $k_1m_1$ or $k_2m_2$ is odd. Hence $k_1m_1+k_2m_2\geq 3\times 3+ 2\times 2= 13$, a contradiction.
\end{proof}

 Let us assume without loss of generality that $|O_1|<|O_2|$. In the next three propositions we consider all possibilities for the sizes of $O_1$, leading to the conclusion that  $B$ must be primitive. Note that this is also true when $|O_1|=4$, as in this case $|O_2|$ equals $7$, a prime number. So in what follows we need only to consider  $|O_1|=2$,  $|O_1|=3$ and $|O_1|=5$.


\begin{prop}\label{B2}
If $|O_1|=2$ then $B$ is primitive.
\end{prop}
\begin{proof}
Suppose that $B\leq S_{k_2}\wr S_{m_2}$ with $k_2=m_2=3$. As the $i$-split does not belong to a square, $J_A\subseteq \{i-1,i+1\}$.
Suppose first that $J_B$ is not an interval. Both $\rho_{i-1}$ and $\rho_{i+1}$ must act nontrivially on $O_2$. Then $\mathcal{G}$ must contain the following graph.
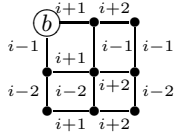
\begin{small}
$$\xymatrix@-1.1pc{  *+[o][F]{b} \ar@{-}[r]^{i+1} \ar@{-}[d]_{i-1} & *{\bullet}\ar@{-}[r]^{i+2}  \ar@{-}[d]^{i-1}&*{\bullet}\ar@{-}[d]^{i-1}\\
 *{\bullet} \ar@{-}[d]_{i-2}\ar@{-}[r]^{i+1} &*{\bullet}\ar@{-}[r]_{i+2} \ar@{-}[d]_{i-2} &*{\bullet}\ar@{-}[d]^{i-2}\\
 *{\bullet}\ar@{-}[r]_{i+1}  & *{\bullet}\ar@{-}[r]_{i+2} &*{\bullet} }$$
 \end{small}
Either $\rho_{i-2}$ or $\rho_{i+2}$ fixes $O_1$, thus $G$ contains a $3$-transposition, a contradiction.
Thus $J_B$ must be an interval. Suppose, up to duality, that all labels in $J_B$ are greater than $i$. If $i\neq 0$ then, as $0\notin J_B$,  $\rho_0$ is odd, a contradiction. 
Thus $i=0$ and, by Proposition~\ref{path} we necessarily have $J_A=\{1\}$, thus the permutation graph has the following subgraph. 
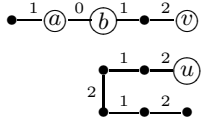
\begin{small}
$$\xymatrix@-1.3pc{  *{\bullet}\ar@{-}[r]^{1} & *+[o][F]{a} \ar@{-}[r]^{0} & *+[o][F]{b} \ar@{-}[r]^{1} & *{\bullet} \ar@{-}[r]^{2} & *+[o][F]{v}\\
&&*{\bullet}\ar@{-}[d]_2\ar@{-}[r]^{1} & *{\bullet} \ar@{-}[r]^{2} & *+[o][F]{u} \\
&&*{\bullet}\ar@{-}[r]^{1} & *{\bullet} \ar@{-}[r]^{2} & *{\bullet}}$$
\end{small}
Moreover, without loss of generality, the permutation representation graph of $G$ has an edge $\{u,v\}$ which either has label $1$ or $3$. If it is $1$ then $\rho_1$ is a $5$-transposition, a contradiction.
If the edge $\{u,v\}$ has label $3$ then $G_0$ has one of the following permutation representation subgraphs.
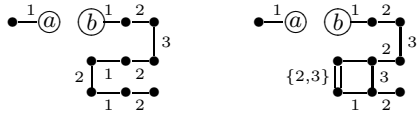
\begin{small}
$$\xymatrix@-1.5pc{  *{\bullet}\ar@{-}[r]^{1} & *+[o][F]{a}  & *+[o][F]{b} \ar@{-}[r]^{1} & *{\bullet} \ar@{-}[r]^{2} & *{\bullet}\\
&&*{\bullet}\ar@{-}[d]_2\ar@{-}[r]_{1} & *{\bullet} \ar@{-}[r]_{2} & *{\bullet}\ar@{-}[u]_{3}\\
&&*{\bullet}\ar@{-}[r]_{1} & *{\bullet} \ar@{-}[r]_{2} & *{\bullet}}\quad\quad\quad \xymatrix@-1.5pc{  *{\bullet}\ar@{-}[r]^{1} & *+[o][F]{a}  & *+[o][F]{b} \ar@{-}[r]^{1} & *{\bullet} \ar@{-}[r]^{2} & *{\bullet}\\
&&*{\bullet}\ar@{=}[d]_{\{2,3\}}\ar@{-}[r]{1} & *{\bullet}\ar@{-}[d]^3\ \ar@{-}[r]^{2} & *{\bullet}\ar@{-}[u]_{3}\\
&&*{\bullet}\ar@{-}[r]_{1} & *{\bullet} \ar@{-}[r]_{2} & *{\bullet}}$$
\end{small}
In any case $\rho_3$ is odd, a contradiction.
With this we conclude that if  $|O_1|=2$, then  $B$ is primitive, as required.
\end{proof}

\begin{prop}\label{B3}
If $|O_1|=3$ then  $B$ is primitive. 
\end{prop}
\begin{proof}
Suppose that $B$ is an imprimitive permutation group (of degree $8$).
Let us assume without loss of generality that $J_A=\{i+1,i+2\}$. Then $\{i+1,i+2\}\subseteq J_B$ otherwise $\rho_{i+1}$ and $\rho_{i+2}$ are odd.
Suppose that $J_B$ is not an interval, then $i\in\{1,2\}$.  If $i=1$, then $J_A=\{2,3\}$ and,  as $\langle \rho_0\rangle$ is an intransitive normal subgroup of $B$, $\rho_0$ determines a block system for the group $B$, with four blocks of size two.
As $G_0$ is intransitive, $\rho_0$ is the unique permutation acting non-trivially within the blocks. This forces $\rho_2$ and $\rho_3$ to be odd permutations.
Thus $i=2$ and $J_A=\{3,4\}$. Notice that $\langle\rho_0,\rho_1\rangle$ cannot be transitive on $O_2$, otherwise $\rho_0$ is odd. Thus, the orbits of $\langle\rho_0,\rho_1\rangle$ determine a block system and necessarily have size greater than $2$. Then $k_2=4$ and $m_2=2$. Then $\rho_3$ must be the unique permutation swapping the two blocks of size four. But then as  $\rho_3$  commutes with $\rho_1$, $\rho_3$ is odd, a contradiction.

Thus $J_B$ is an interval. Therefore $i=0$ and $m_2=4$, otherwise $\rho_1$ is odd.
We find the following possibilities for $G_0$  with $\rho_2$ and $\rho_3$ being even. 
For $r=5$:
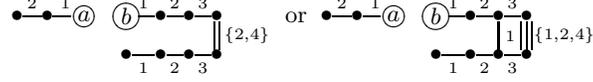
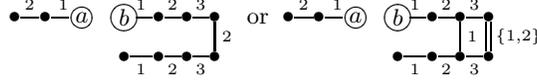
\begin{small}
$$\xymatrix@-1.5pc{  *{\bullet}\ar@{-}[r]^{2} & *{\bullet}\ar@{-}[r]^{1} & *+[o][F]{a}  & *+[o][F]{b} \ar@{-}[r]^{1} & *{\bullet} \ar@{-}[r]^{2}  & *{\bullet}\ar@{-}[r]^{3}  & *{\bullet}\ar@{=}[d]^{\{2,4\}}\\
 &  & & *{\bullet}\ar@{-}[r]_{1} & *{\bullet} \ar@{-}[r]_{2} & *{\bullet}\ar@{-}[r]_{3} & *{\bullet}}
 \mbox{ or }\,
 \xymatrix@-1.5pc{  *{\bullet}\ar@{-}[r]^{2} & *{\bullet}\ar@{-}[r]^{1} & *+[o][F]{a}  & *+[o][F]{b} \ar@{-}[r]^{1} & *{\bullet} \ar@{-}[r]^{2}  & *{\bullet}\ar@{-}[r]^{3} \ar@{-}[d]^{1} & *{\bullet}\ar@3{-}[d]^{\{1,2,4\}}\\
 &  & & *{\bullet}\ar@{-}[r]_{1} & *{\bullet} \ar@{-}[r]_{2} & *{\bullet}\ar@{-}[r]_{3} & *{\bullet}}$$
 For $r=4$:
 $$\xymatrix@-1.5pc{  *{\bullet}\ar@{-}[r]^{2} & *{\bullet}\ar@{-}[r]^{1} & *+[o][F]{a}  & *+[o][F]{b} \ar@{-}[r]^{1} & *{\bullet} \ar@{-}[r]^{2}  & *{\bullet}\ar@{-}[r]^{3}  & *{\bullet}\ar@{-}[d]^2\\
 &  & & *{\bullet}\ar@{-}[r]_{1} & *{\bullet} \ar@{-}[r]_{2} & *{\bullet}\ar@{-}[r]_{3} & *{\bullet}}
 \mbox{ or }\,
 \xymatrix@-1.5pc{  *{\bullet}\ar@{-}[r]^{2} & *{\bullet}\ar@{-}[r]^{1} & *+[o][F]{a}  & *+[o][F]{b} \ar@{-}[r]^{1} & *{\bullet} \ar@{-}[r]^{2}  & *{\bullet}\ar@{-}[r]^{3} \ar@{-}[d]^{1} & *{\bullet}\ar@{=}[d]^{\{1,2\}}\\
 &  & & *{\bullet}\ar@{-}[r]_{1} & *{\bullet} \ar@{-}[r]_{2} & *{\bullet}\ar@{-}[r]_{3} & *{\bullet}}$$
 \end{small}
 When $r=5$,  $G_4$ is transitive, $\rho_4$ is odd and $\rho_1$ is odd. When $r=4$, $\rho_1$ is odd. In any case we have a contradiction.
\end{proof}
\begin{prop}\label{B5}
If $|O_1|= 5$  then  $B$ is primitive.
\end{prop}
\begin{proof}
As $|O_1|$ is prime, $A$ is primitive. By Proposition~\ref{CCD} we may assume that all labels in $J_A$ are greater than $i$.
Now assume that $B$ is embedded into $S_{k_2}\wr S_{m_2}$ with $k_2,\,m_2>1$ and $k_2m_2=6$.

\underline{Suppose first that $m_2=2$ and $k_2=3$.} The $i$-split $\{a, b\}$ does not belong to a square, hence the permutation swapping the blocks is either $\rho_{i-1}$ or $\rho_{i+1}$. If it is $\rho_{i-1}$ then, as  $\rho_{i-1}$ fixes $O_1$, we have that $\rho_{i-1}$  is an odd permutation, a contradiction.
Thus $\rho_{i+1}$ is the unique  permutation swapping the blocks. Thus $\rho_{i+1}$ acts as an odd permutation in both orbits, that is, both $\alpha_{i+1}$ and $\beta_{i+1}$ are odd. Then $A$ cannot be the even group $D_{10}$, hence $i+3\in J_A$.
By Proposition~\ref{path} $\alpha_{i+3}$ must be odd, hence $i+3\in J_B$ and $\beta_{i+3}$ is odd. But $\rho_{i+3}$ commutes with $\rho_{i+1}$ (the permutation swapping the blocks), which forces $\beta_{i+3}$ to be an even permutation, a contradiction.
 
 \underline{Now suppose that $m_2=3$ and $k_2=2$.}
 If $J_B$ is not an interval then $B\cong C_2\times S_3$, particularly $B$ is also embedded into $S_3\wr C_2$. We have just concluded that  this case leads to a contradiction.
 Thus $J_B$ is an interval.  If the labels  of $J_B$ are smaller than $i$ then $r=5$ and  $i=2$. Then there is only one possibility for the permutation representation graph of $G_2$ on the orbit $O_2$, which is as follows.
 
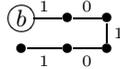
\begin{small}
 $$\xymatrix@-1.8pc{*+[o][F]{b} \ar@{-}[rr]^1&& *{\bullet}\ar@{-}[rr]^0&& *{\bullet}\ar@{-}[d]^1\\
*{\bullet}\ar@{-}[rr]_1&& *{\bullet}\ar@{-}[rr]_0 && *{\bullet}}$$
\end{small}
Then  $\rho_1$ is odd, a contradiction.
If all labels in $J_B$ are greater than $i$, then $i=0$. 
Then there are three possible permutation representation graphs of $G_0$ on the orbit $O_2$. 
\begin{small}
$$\xymatrix@-1.8pc{*+[o][F]{b} \ar@{-}[rr]^1&& *{\bullet}\ar@{-}[rr]^2&& *{\bullet}\ar@{-}[d]^1\\
*{\bullet}\ar@{-}[rr]_1&& *{\bullet}\ar@{-}[rr]_2 && *{\bullet}}\quad \xymatrix@-1.8pc{*+[o][F]{b} \ar@{-}[rr]^1&& *{\bullet}\ar@{-}[rr]^2&& *{\bullet}\ar@{-}[d]^3\\
*{\bullet}\ar@{-}[rr]_1&& *{\bullet}\ar@{-}[rr]_2 && *{\bullet}}\quad \xymatrix@-1.8pc{*+[o][F]{b} \ar@{-}[rr]^1&& *{\bullet}\ar@{-}[rr]^2&& *{\bullet}\ar@{=}[d]^{\{1,3\}}\\
*{\bullet}\ar@{-}[rr]_1&& *{\bullet}\ar@{-}[rr]_2 && *{\bullet}}
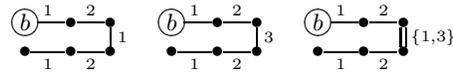
 $$
 \end{small}
As any path  in $O_1$ containing two $3$-edges has at least six vertices, the possibility on the left can be excluded. In the other cases we get that either $\rho_0$ or $\rho_3$ are odd, a contradiction.

\end{proof}
Now let us consider the case when $B$ is primitive. By Proposition~\ref{CCD}, $J_B$ is an interval.
Let us assume, without loss of generality, that any label in $J_B$ is greater than $i$.

\begin{prop}\label{O2}
$|O_1|\neq 2$         					
\end{prop}
\begin{proof}
Suppose that $|O_1|=2$.
If  $i\neq 0$ then $i-1\in J_A$ and  $i-1\notin J_B$. This implies that $\rho_{i-1}$ is odd, a contradiction.
Hence $i=0$ and $J_A=\{1\}$. 
 Let us now prove that $\rho_1$ is a $4$-transposition. Suppose that $\rho_1$ is a $2$-transposition. Then the orbit of  $G_{>1}$ containing  $b\rho_1$ has $8$ vertices. As $\rho_0$ must act non-trivially on this orbit and $\rho_0$ centralizes $G_{>1}$, $\rho_0$ is a $5$-transposition, a contradiction. Indeed, $G_{>1}$ needs to have more than two orbits on $O_2$.

Let us now prove that $\rho_3$ is a $2$-transposition. Suppose that $\rho_3$ is a $4$-transposition. 
As the shortest path from $b$ including the first 3-edge must have four vertices, by Proposition~\ref{path}, this implies that the size of $|O_2|\geq 4+3\times 2=10$ a contradiction.  
 Let us now consider the cases $r=4$ and $r=5$ separately.
 
\underline{$r=5$:}  By Proposition~\ref{path} $\rho_4$ must be a $2$-transposition. Suppose that $\mathcal{G}$ has a $\{2,4\}$-square. As observed before, $G_{>1}$ has more than two orbits on $O_2$.  Hence a path from $b$ to the $\{2,4\}$-square must contain two $1$-edges. This gives the following possibility for $\mathcal{G}_0$.
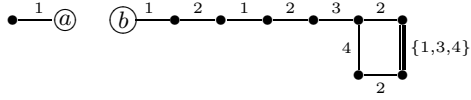
\begin{small}
$$\xymatrix@-1pc{  *{\bullet}\ar@{-}[r]^{1} & *+[o][F]{a}  & *+[o][F]{b} \ar@{-}[r]^{1} & *{\bullet} \ar@{-}[r]^{2}  & *{\bullet}\ar@{-}[r]^{1}  & *{\bullet}\ar@{-}[r]^2 &*{\bullet}\ar@{-}[r]^3 & *{\bullet} \ar@{-}[r]^{2} \ar@{-}[d]_{4}  & *{\bullet}\ar@{-}[d] \ar@{=}[d]^{\{1,3,4\}} \\
&&&&&&& *{\bullet}\ar@{-}[r]_{2} &*{\bullet}}
$$
\end{small}
But then $G_4$ is transitive, a contradiction.
We get the same contradiction if we admit that the graph has $\{2,4\}$-edges. Hence,  $\rho_2$ is a $2$-transposition.

As $\rho_1$ is a $4$-transposition and $\rho_4$ is a $2$-transposition, $\mathcal{G}$ cannot contain a $\{1,4\}$-square. Suppose that  $\mathcal{G}$  contains a $\{1,4\}$-square, then as $\rho_4$ is a $2$-transposition, there exist a $3$-edge incident to exactly one vertex of the $\{1,4\}$-square. Then  $\mathcal{G}$ has  a $\{1,3\}$-square sharing an $1$-edge with the  $\{1,4\}$-square. Moreover, the vertex  $b$, of the split,  cannot be a vertex of that $\{1,3\}$-square. This implies that the graph has at least five $1$-edges, a contradiciton.

Thus an edge adjacent to a $4$-edge must have label $3$. Thus $\mathcal{G}$ contains the following graph.
$$\xymatrix@-1pc{  *{\bullet}\ar@{-}[r]^{1} & *+[o][F]{a}  & *+[o][F]{b} \ar@{-}[r]^{1} & *{\bullet} \ar@{-}[r]^{2}  & *{\bullet}\ar@{-}[r]^{1}  & *{\bullet}\ar@{-}[r]^2 &*{\bullet}\ar@{-}[r]^3 & *{\bullet} \ar@{-}[r]^{4}  & *{\bullet} \ar@{-}[r]^{3}  & *{\bullet}\ar@{-}[r]^{4}  & *{\bullet}} 
$$
 This implies that $\rho_1$ is a $3$-transposition, a contradiction.
 
\underline{$r=4$:}  As $\rho_3$ is a $2$-transposition, for connectedness, $\rho_2$ must be a $4$-transposition, that means that $\rho_2$ is fixed-point-free in $O_2\setminus\{b\}$.
As $\rho_0$ must be odd in $O_2$,  $\mathcal{G}$ must have a $\{0,2\}$-edge. If it is adjacent to a $3$-edge we get the following possibility with $G_3$ being transitive, a contradiction.

$$
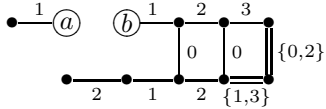
\xymatrix@-1pc{  *{\bullet}\ar@{-}[r]^{1} & *+[o][F]{a}  & *+[o][F]{b} \ar@{-}[r]^{1} & *{\bullet} \ar@{-}[r]^{2}\ar@{-}[d]^{0}  & *{\bullet}\ar@{-}[r]^{3}\ar@{-}[d]^{0}    & *{\bullet} \ar@{=}[d]^{\{0,2\}} \\
&*{\bullet}\ar@{-}[r]_{2}  &*{\bullet}\ar@{-}[r]_1  &*{\bullet} \ar@{-}[r]_{2} & *{\bullet}\ar@{=}[r]_{\{1,3\} }&*{\bullet}}
$$
Thus the $\{0,2\}$-edge is not adjacent to a $3$-edge.
If $\mathcal{G}$ contains a $\{0,2\}$-square, then we can determine the components of $G_{1,3}$ in $O_2\setminus\{b\}$: a $\{0,2\}$-edge, a $\{0,2\}$-square and a $2$-edge. Thus, the 3-edges cannot connect the components above. Hence, $G_3$ is transitive, a contradiction.
By the same reason $\mathcal{G}$  has exactly one $\{0,2\}$-edge.
Thus $\rho_0$ is a $2$-transposition.
Note also that if  $\mathcal{G}$ contains a $\{1,3\}$-square then the  $\{0,2\}$-edge must share at least one vertex with that square, otherwise, since $\rho_1$ is a $4$-transposition, the graph is disconnected. As $\rho_3$ is a $2$-transposition, both vertices of the $\{0,2\}$-edge belong to the $\{1,3\}$-square. But in this case $G_3$ is transitive, a contradiction.
 Thus $\mathcal{G}$ does not contain a  $\{1,3\}$-square.
This gives only the four possibilities corresponding to graphs (A1) to (A4) in the appendix, which are not string C-groups, a contradiction.

\end{proof}

\begin{prop}\label{O3}
 $|O_1|\neq 3$  
\end{prop}
\begin{proof} Suppose that $|O_1|=3$. 
In this case $J_A$ and $J_B$ are both intervals.
If $i\neq 0$ then $A=G_{<i}$ and $B=G_{>i}$. 
Moreover $G_{<i}$ is an even string C-group of degree $3$, but a string C-group of degree $3$ is isomorphic to $S_3$, a contradiction.

Thus $i=0$ and $J_A=\{1,2\}$. Moreover the $O_1$ component of $\mathcal{G}_0$ is a path. Let $O_1=\{1,2,3\}$ and  $\{a,b\}=\{3,4\}$ as is the following figure.
\begin{small}
$$\xymatrix@-1.3pc{*+[o][F]{1}  \ar@{-}[rr]^2&& *+[o][F]{2} \ar@{-}[rr]^1&&*+[o][F]{3} \ar@{.}[rr]^0&& *+[o][F]{4} \ar@{-}[rr]^1&&*+[o][F]{5} \ar@{-}[rr]^2&& *+[o][F]{6} \ar@{.}[rr]&&  } $$
\end{small}

Suppose first that $G_{0,1}$ has exactly four orbits: $\{1,2\},\{3\},\{4\}$ and $\{5,\ldots,11\}$.
 In this case $\Gamma_{0,1}$ is a sesqui-extension with respect to $\rho_2$ of string C-group $\Lambda$ acting transitively on $11-4=7$ points.
Moreover, as $(1,2)\notin G_{0,1}$, by Lemma~\ref{sesqui}, $\Lambda$ is a string group representation of a group isomorphic to $G_{0,1}$ and $\Lambda$ must have an index $2$ subgroup (which is the even subgroup of $\Lambda$ of the elements that can be written with an even number of $\rho_2$'s).
Then there is only one possibility which is $G_{0,1}\cong S_7$. In particular $r=5$.

By Proposition~\ref{path}, note that $\rho_3$ and $\rho_4$ must be $2$-transpositions. But then for connectedness of  $\mathcal{G}$, $\rho_2$ must be a $4$-transposition.
Then $\mathcal{G}$ has an $\{2,4\}$-square. This gives the  possibilities (B1), (B2) and (B3) of the appendix which are not string C-groups by Proposition~\ref{notC}.

Now suppose that $G_{0,1}$ has orbits:  $\{1,2\},\{3\},\{4\}$, $\{5,6\}$ and $\{7,\ldots,11\}$.
By Proposition~\ref{path}, $r=4$ and $\mathcal{G}$ must have the following spanning subgraph, where $G_{0,1}\cong D_{20}$.
\begin{small}
$$ \xymatrix@-1.6pc{*+[o][F]{1}  \ar@{-}[rr]^2 && *+[o][F]{2} \ar@{-}[rr]^1&&*+[o][F]{3} \ar@{.}[rr]^0&& *+[o][F]{4} \ar@{-}[rr]^1&&*+[o][F]{5} \ar@{-}[rr]^2&& *+[o][F]{6} \ar@{-}[rr]^1&& *+[o][F]{7} \ar@{-}[rr]^2&&*+[o][F]{8} \ar@{-}[rr]^3&&*+[o][F]{9} \ar@{-}[rr]^2&& *+[o][F]{10} \ar@{-}[rr]^3&&*+[o][F]{11} }$$
\end{small}
Now $\mathcal{G}$ must have precisely one more $1$-edge and at least one more $0$-edge. 
By the commuting property $\rho_0$ must fix the vertices $\{7,8,9,10,11 \}$, so
 there are only two possibilities for the other $0$-edge, $\{1,2\}$ or $\{5,6\}$. 
In addition there are also two possibilities for the other $1$-edge, $\{8,9\}$ or $\{10,11\}$. Hence $\mathcal{G}$ is one of the graphs (B4), (B5), (B6) or (B7) of the appendix.

Now suppose that the orbit of the vertex $5$, in $G_{0,1}$ has more than two points (and less than $7$) then $\rho_3$ does not fix the vertex $6$.
If both $\rho_1$ and $\rho_3$ act non-trivially on the vertex $6$ then there is either a $\{1,3 \}$-square or a double $\{1,3\}$-edge. In total this yields six possibilities for $\mathcal{G}$, the graphs (B8)-(B13) of the appendix.

If $\rho_1$ fixes the vertex $6$, then we get the graphs  (B14) and (B15) of the appendix.

Thus, in every case, $\Gamma$ is not a string C-group by Proposition~\ref{notC}, a contradiction. 
 
\end{proof}


\begin{prop}\label{O4}
$|O_1|\neq 4$  
\end{prop}
\begin{proof}
Suppose $|O_1|=4$. Now suppose that $J_A$ is not an interval.  In this case $A$ is imprimitive with two blocks of size $2$. As only edges with labels $i\pm 1$ are incident to $a$, and remembering that $J_B$ is an interval with labels greater than $i$, there are only two possibilities either $J_A=\{0,2\}$ or $J_A=\{0,2,3\}$. In any case $i=1$ and the permutation representation  of $G_1$ restricted to $O_1$ is one of the following graphs.
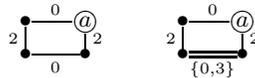
\begin{small}
$$\xymatrix@-1.6pc{*{\bullet} \ar@{-}[rr]^0 \ar@{-}[d]_2&& *+[o][F]{a} \ar@{-}[d]^2 \\
*{\bullet} \ar@{-}[rr]_0&& *{\bullet} }\quad \quad \xymatrix@-1.6pc{*{\bullet} \ar@{-}[rr]^0 \ar@{-}[d]_2&& *+[o][F]{a} \ar@{-}[d]^2 \\
*{\bullet} \ar@{=}[rr]_{\{0,3\}}&& *{\bullet} }$$
\end{small}
Recalling that, by Proposition~\ref{B1}, $J_B$ is an interval, we may exclude the second possibility, for otherwise $G_0$ is transitive.
Thus we have to consider the case $J_A=\{0,2\}$ with $A$ having the permutation representation graph on the left.

Now, if $J_A$ is an interval then $i=0$ and there are only the following three possibilities for the permutation representation of $G_0$ in the orbit $O_1$.
\begin{small}
$$\xymatrix@-1.5pc{*{\bullet} \ar@{-}[rr]^1&& *{\bullet}\ar@{-}[rr]^2 &&*{\bullet} \ar@{-}[rr]^1&& *+[o][F]{a} }\quad \xymatrix@-1.5pc{*{\bullet} \ar@{=}[rr]^{\{1,3\}}&& *{\bullet}\ar@{-}[rr]^2&& *{\bullet} \ar@{-}[rr]^1&& *+[o][F]{a} }\quad  \xymatrix@-1.5pc{*{\bullet} \ar@{-}[rr]^3&& *{\bullet}\ar@{-}[rr]^2&& *{\bullet} \ar@{-}[rr]^1&& *+[o][F]{a} }$$
\end{small}
Hence, we have to consider the three cases above  (when $i=0$) plus the  $\{0,2\}$-square when $i=1$. Let us prove that in any case $r=4$.
Suppose that $r>4$. Then $4\in J_B$ and the labelling set of any path from $b$ to a $4$-edge must contain the set  $\{i+1,\ldots,3\}$.
As $4\notin J_A$ and $|O_2|=7$, $\mathcal{G}$ must contain one of the following two paths with $7$ vertices of $O_2$ having at least two 4-edges. In the first $i=1$ and the second $i=0$.
\begin{small}
$$\xymatrix@-1.6pc{*+[o][F]{b}  \ar@{-}[rr]^2&& *{\bullet}\ar@{-}[rr]^3 &&*{\bullet} \ar@{-}[rr]^4&& *{\bullet}\ar@{-}[rr]^3&& *{\bullet}\ar@{-}[rr]^4&& *{\bullet} \ar@{-}[rr]^3&& *{\bullet}}\quad\xymatrix@-1.6pc{*+[o][F]{b}  \ar@{-}[rr]^1&& *{\bullet}\ar@{-}[rr]^2 &&*{\bullet} \ar@{-}[rr]^3&& *{\bullet}\ar@{-}[rr]^4&& *{\bullet}\ar@{-}[rr]^3&& *{\bullet} \ar@{-}[rr]^4&& *{\bullet}}
 $$
 \end{small}
 In the first case $\rho_3$ is odd. In the second case $\rho_0$ fixes $O_2$ pointwisely, hence it must act non-trivially on $O_1$. Thus we must have $J_A=\{1,2\}$, but then $\rho_1$ is odd, a contradiction.
 Thus $r=4$.
 
 We now deal separately with the cases $J_A=\{0,2\}$,  $J_A=\{1,2\}$ and $J_A=\{1,2,3\}$.

 If $J_A=\{0,2\}$ then $\langle \rho_2,\rho_3\rangle$ must be transitive in $O_2$, which has size $7$. This implies that both $\rho_2$ and $\rho_3$ are odd, a contradiction.

If  $J_A=\{1,2\}$ then $\rho_3$ fixes $O_1$, and the commuting property forces  a path starting at the vertex $b$ and  containing two $3$-edges,  to have exactly  $6$ vertices. Having in mind that $G$ is even, we find the sggi which is the permutation graph (C1) of the appendix and the following two permutation representation graphs.
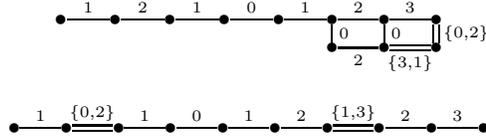
\begin{small}
$$\xymatrix@-1.6pc{*{\bullet} \ar@{-}[rr]^1&& *{\bullet}\ar@{-}[rr]^2 &&*{\bullet} \ar@{-}[rr]^1&&*{\bullet} \ar@{-}[rr]^0&&*{\bullet} \ar@{-}[rr]^1&& *{\bullet}\ar@{-}[rr]^2 \ar@{-}[d]^0&& *{\bullet}\ar@{-}[rr]^3 \ar@{-}[d]^0&& *{\bullet}\ar@{=}[d]^{\{0,2\}}\\
&& && &&  && &&*{\bullet} \ar@{-}[rr]_2&& *{\bullet}\ar@{=}[rr]_{\{3,1\}} &&*{\bullet}}
 $$

  $$\xymatrix@-1.6pc{*{\bullet} \ar@{-}[rr]^1 && *{\bullet}\ar@{=}[rr]^{\{0,2\}}&&*{\bullet}\ar@{-}[rr]^1&& *{\bullet}\ar@{-}[rr]^0&&*{\bullet}\ar@{-}[rr]^1&& *{\bullet}\ar@{-}[rr]^2&& *{\bullet}\ar@{=}[rr]^{\{1,3\}}&&*{\bullet}\ar@{-}[rr]^2&&*{\bullet}\ar@{-}[rr]^3&& *{\bullet}\ar@{-}[rr]^2&&*{\bullet}}$$
  \end{small}
 If $\mathcal{G}$ is the first  graph $G_3$ is transitive, a contradiction.
 If  $\mathcal{G}$  is the second graph above then it has a split with label $3$ and $G_3$ has an orbit of size two, contradicting Proposition~\ref{O2}.

If $J_A=\{1,2,3\}$ then $\rho_3$ will swap exactly one pair of vertices of $O_2$. Indeed since $i=0$ and $|O_2|=7$, $\rho_3$ cannot swap three pairs of vertices of $O_2$.
Consider the minimal path, in $\mathcal{G}$, starting at $b$ and containing the $3$-edge of $O_2$.
This path must  have  $4$ or $6$ vertices, thanks to the commuting property. 
If it has $4$ vertices we get the sggi  (C2), (C3)  or (C4) of the appendix. 
 If it has $6$ vertices then we get either the  sggi (C5) or (C6) of the appendix,
or the following graph which may be dismissed by Proposition~\ref{O2} since it has a $3$-split with one orbit of size two.
\begin{small}
  $$\xymatrix@-1.8pc{*{\bullet} \ar@{=}[rr]^{\{1,3\}}&& *{\bullet}\ar@{-}[rr]^2&&*{\bullet}\ar@{-}[rr]^1&& *{\bullet}\ar@{-}[rr]^0&&*{\bullet}\ar@{-}[rr]^1&& *{\bullet}\ar@{=}[rr]^{\{0,2\}}&& *{\bullet}\ar@{-}[rr]^1&&*{\bullet}\ar@{-}[rr]^2&&*{\bullet}\ar@{-}[rr]^3&& *{\bullet}\ar@{-}[rr]^2&&*{\bullet}}$$ 
\end{small}
The sggis of the appendix do not satisfy the intersection property and so, in all of the remaining cases, $\Gamma$ is not a string C-group by Proposition~\ref{notC}, a contradiction.

\end{proof}
 
\begin{prop}\label{O5}
$|O_1|\neq 5$ 
\end{prop}
\begin{proof}
Since $|O_1|=5$, then $A$ is primitive and thus $J_A$ is an interval. Suppose $i\neq 0$. Then since $J_B$ only has labels greater than $i$, all the labels of $J_A$ are all smaller than $i$. Therefore $A=G_{<i}$ and $B=G_{>i}$. This implies that $r=5$ and the groups $A$ and $B$ must be dihedral. But $D_{12}$ is odd, a contradiction. Consequently $i=0$.
Now if the rank is $5$ then the permutation graph of $\Gamma$ contains the following path of size $10$.
\begin{small}
$$\xymatrix@-1.8pc{*{\bullet}\ar@{-}[rr]^4&& *{\bullet}\ar@{-}[rr]^3&& *{\bullet}\ar@{-}[rr]^2&&*{\bullet}\ar@{-}[rr]^1&& *{\bullet}\ar@{-}[rr]^0&& *{\bullet}\ar@{-}[rr]^1&&*{\bullet}\ar@{-}[rr]^2&& *{\bullet}\ar@{-}[rr]^3&& *{\bullet}\ar@{-}[rr]^4&& *{\bullet}}$$
\end{small}
This graph cannot be a subgraph of $\mathcal{G}$, as this forces $\rho_0$ to be odd.
Thus $r\leq 4$. As $G$, is even $r=4$.

Suppose first that $\rho_3$ fixes $O_1$  pointwisely. Then we get the following permutation representation graphs of $\mathcal{G}_0$.
\begin{small}
$$\xymatrix@-1.8pc{*{\bullet}\ar@{-}[rr]^2&& *{\bullet}\ar@{-}[rr]^1&& *{\bullet}\ar@{-}[rr]^2&&*{\bullet}\ar@{-}[rr]^1&& *+[o][F]{a} && *+[o][F]{b} \ar@{-}[rr]^1&&*{\bullet}\ar@{-}[rr]^2&& *{\bullet}\ar@{-}[rr]^3&& *{\bullet}\ar@{-}[rr]^2&&*{\bullet}\ar@{=}[rr]^{\{1,3\}}&& *{\bullet}}$$
\end{small}
or
\begin{small}
$$\xymatrix@-1.8pc{*{\bullet}\ar@{-}[rr]^2&& *{\bullet}\ar@{-}[rr]^1&& *{\bullet}\ar@{-}[rr]^2&&*{\bullet}\ar@{-}[rr]^1&& *+[o][F]{a} && *+[o][F]{b} \ar@{-}[rr]^1&&*{\bullet}\ar@{-}[rr]^2&& *{\bullet}\ar@{=}[rr]^{\{1,3\}}&& *{\bullet}\ar@{-}[rr]^2&&*{\bullet}\ar@{-}[rr]^3&& *{\bullet}}$$
\end{small}
In both cases $3$ is the label of a split and $G_3$ has one orbit of size $3$ or $1$, respectively, and so by Proposition~\ref{O3} we may exclude the first of these graphs.
From the second graph we get the graphs (D1) and (D2) of the appendix.

Now consider that $\rho_3$ has a non-trivial action in $O_1$. In this case  $\Gamma_0$ has one of the following permutation representation graphs, giving graphs (D3) and (D4) of the appendix.
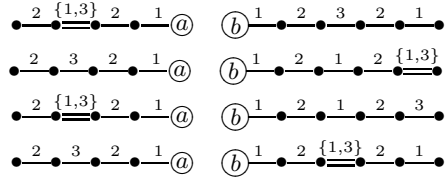
\begin{small}
\[\begin{array}{c}
\xymatrix@-1.8pc{*{\bullet}\ar@{-}[rr]^2&& *{\bullet}\ar@{=}[rr]^{\{1,3\}}&& *{\bullet}\ar@{-}[rr]^2&&*{\bullet}\ar@{-}[rr]^1&& *+[o][F]{a} && *+[o][F]{b} \ar@{-}[rr]^1&&*{\bullet}\ar@{-}[rr]^2&& *{\bullet}\ar@{-}[rr]^3&& *{\bullet}\ar@{-}[rr]^2&&*{\bullet}\ar@{-}[rr]^1&& *{\bullet}}\\
\xymatrix@-1.8pc{*{\bullet}\ar@{-}[rr]^2&& *{\bullet}\ar@{-}[rr]^3&& *{\bullet}\ar@{-}[rr]^2&&*{\bullet}\ar@{-}[rr]^1&& *+[o][F]{a} && *+[o][F]{b} \ar@{-}[rr]^1&&*{\bullet}\ar@{-}[rr]^2&& *{\bullet}\ar@{-}[rr]^1&& *{\bullet}\ar@{-}[rr]^2&&*{\bullet}\ar@{=}[rr]^{\{1,3\}}&& *{\bullet}}\\
\xymatrix@-1.8pc{*{\bullet}\ar@{-}[rr]^2&& *{\bullet}\ar@{=}[rr]^{\{1,3\}}&& *{\bullet}\ar@{-}[rr]^2&&*{\bullet}\ar@{-}[rr]^1&& *+[o][F]{a} && *+[o][F]{b} \ar@{-}[rr]^1&&*{\bullet}\ar@{-}[rr]^2&& *{\bullet}\ar@{-}[rr]^1&& *{\bullet}\ar@{-}[rr]^2&&*{\bullet}\ar@{-}[rr]^3&& *{\bullet}}\\
\xymatrix@-1.8pc{*{\bullet}\ar@{-}[rr]^2&& *{\bullet}\ar@{-}[rr]^3&& *{\bullet}\ar@{-}[rr]^2&&*{\bullet}\ar@{-}[rr]^1&& *+[o][F]{a} && *+[o][F]{b} \ar@{-}[rr]^1&&*{\bullet}\ar@{-}[rr]^2&& *{\bullet}\ar@{=}[rr]^{\{1,3\}}&& *{\bullet}\ar@{-}[rr]^2&&*{\bullet}\ar@{-}[rr]^1&& *{\bullet}}
\end{array}\]
\end{small}
As before, there is a 3-split for the first, second and fourth graphs, where $G_3$ has either an orbit of size two or three.
By Propositions~\ref{O2} and~\ref{O3}  all the possibilities for $\mathcal{G}$ are given in the appendix. In any case we have a contradiction with the intersection property by Proposition~\ref{notC}.

\end{proof}

The case where the connected components of $\mathcal{G}_i$ are nontrivial is now completed and the conclusion is the following.

\begin{prop}\label{all}
Let $\Gamma$ be an even string C-group of degree $11$ and rank $r\in\{3,4,5\}$. 
Suppose that $\Gamma$ has a fracture graph.
If $i$ is the label of a split, then $G_i$ has one trivial orbit. 
\end{prop}
\begin{proof}
This is a consequence of Propositions~\ref{O2}, \ref{O3}, \ref{O4} and \ref{O5}.
\end{proof}


\subsection{Case: $\mathcal{G}_i$ has a trivial component. }

\begin{prop}
If $|O_1|=1$ then $B$ is primitive and $i\in\{0,r-1\}$.
\end{prop}
\begin{proof}
Suppose that $B$ is embedded into $S_{k_2}\wr S_{m_2}$ with $k_2m_2=10$, $k_2,m_2>1$. If $m_2=2$ then the permutation swapping the blocks is a $5$-transposition, a contradiction. Hence $k_2=2$ and $m_2=5$.
First suppose that $i\not\in \{0,r-1\}$.  In that case $B=G_{i}=G_{<i}\times G_{>i}$.
 If $G_{>i}$ is transitive on $O_2$, as $\rho_0$ centralizes $G_{>i}$, $\rho_0$ is fixed-point-free on $O_2$. 
Then $\rho_0$ is a $5$-transposition, a contradiction.  Thus  $G_{>i}$ is intransitive and, by the same argument,  $G_{<i}$ is intransitive.
If either $G_{<i}$ or $G_{>i}$ is a cyclic group, then we have the same contradiction as before.
Hence both groups have two orbits of size $5$. As neither $G_{<i}$ nor $G_{>i}$ is cyclic, we have that $i\notin\{1,r-2\}$.
Consider the blocks of size $5$ corresponding to the $G_{<i}$-orbits.
Then as $\rho_{i+1}$ centralizes $G_{<i}$, it cannot fix the blocks which have odd size, hence \st{both} $\rho_{i+1}$ swaps the blocks.
Similarly $\rho_{i+2}$ swaps the blocks.
Hence $G_{i+1}$ and $G_{i+2}$ are transitive, a contradiction.
Consequently $i\in\{0,r-1\}$.

Without loss of generality lets assume that $i=0$.
Note that a $0$-split $\{a,b\}$ does not belong to a  square, therefore $G_{>1}$ fixes $b$ and $G_{>2}$ fixes $b\rho_1$.
Hence we get the following possibilities for the graph representing the block action of $G_0$.
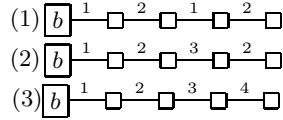
\begin{small}
\[\begin{array}{c}
(1)\,\xymatrix@-0.8pc{  *+[][F]{b} \ar@{-}[r]^1 & *+[][F]{}  \ar@{-}[r]^2 & *+[][F]{} \ar@{-}[r]^1 & *+[][F]{} \ar@{-}[r]^2 & *+[][F]{}}
\\
(2)\,\xymatrix@-0.8pc{  *+[][F]{b} \ar@{-}[r]^1 & *+[][F]{}  \ar@{-}[r]^2 & *+[][F]{} \ar@{-}[r]^3 & *+[][F]{} \ar@{-}[r]^2 & *+[][F]{}}
\\
(3) \xymatrix@-0.8pc{  *+[][F]{b} \ar@{-}[r]^1 & *+[][F]{}  \ar@{-}[r]^2 & *+[][F]{} \ar@{-}[r]^3 & *+[][F]{} \ar@{-}[r]^4 & *+[][F]{}}
\end{array}
\]
\end{small}
But $\rho_0$ must have a non-trivial action on the orbit of size $10$. If $\rho_0$ permutes two vertices in a block, then it permutes another pair of vertices in an adjacent block, which forces $\rho_0$ to be odd.
Thus $\rho_0$ swaps a pair of vertices in different blocks. This is only possible when the block action is as in (1) or (2), corresponding to the following permutation graphs of $G_0$ for $O_2$.
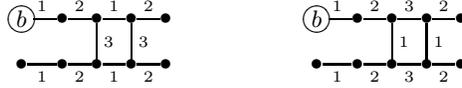
\begin{small}
$$\xymatrix@-1.3pc{  *+[o][F]{b} \ar@{-}[r]^1 & *{\bullet} \ar@{-}[r]^2 & *{\bullet}\ar@{-}[r]^1  \ar@{-}[d]^3& *{\bullet}\ar@{-}[d]^3\ar@{-}[r]^2 & *{\bullet}&&&& *+[o][F]{b} \ar@{-}[r]^1 & *{\bullet} \ar@{-}[r]^2 & *{\bullet}\ar@{-}[r]^3  \ar@{-}[d]^1& *{\bullet}\ar@{-}[d]^1 \ar@{-}[r]^2 & *{\bullet}\\
 *{\bullet}\ar@{-}[r]_1 & *{\bullet} \ar@{-}[r]_2 & *{\bullet}\ar@{-}[r]_1 & *{\bullet}\ar@{-}[r]_2  & *{\bullet}&&&&  *{\bullet}\ar@{-}[r]_1 & *{\bullet} \ar@{-}[r]_2 & *{\bullet}\ar@{-}[r]_3 & *{\bullet}\ar@{-}[r]_2  & *{\bullet}}$$
 \end{small}
By the commuting property it is impossible to place an odd number of $0$-edges into either of the above diagrams.
 Thus $B$ is primitive. By Proposition~ \ref{CCD} $J_B$ must be an interval, hence $i\in\{0,r-1\}$.
 \end{proof}

\begin{prop}\label{not2split}
Let $\Gamma$ be an even string C-group of degree $11$ and rank $r\leq 5$. 
If $\Gamma$ has a fracture graph then either $\mathcal{G}_0$  or  $\mathcal{G}_{r-1}$ has a $2$-fracture graph.
\end{prop}
\begin{proof}
Suppose that neither $\mathcal{G}_0$  nor  $\mathcal{G}_{r-1}$ has a  $2$-fracture graph.
This is only possible if both  $i=0$ and $i=r-1$ are labels of splits. 
 Suppose that $r=4$ in this case $G_{0,3}\cong D_{18}$. Then $G_0$ and $G_3$ are even transitive groups of degree $10$ containing $D_{18}$. 
 Hence, both $G_0$ and $G_3$ contain a $9$-cycle, therefore they are primitive. This is only possible if  $G_0\cong A_{10}$ and $G_3\cong A_{10}$ \cite{ATLAS}. But then  $G_{0,3}\cong A_9$ not $D_{18}$, contradicting the intersection property.
Thus $r=5$.  Let $\{c,d\}$ be the $4$-split.

\underline{$\rho_0$ and $\rho_4$ are $2$-transpositions:}
Suppose that $\rho_0$ is a $4$-transposition. As $\rho_0$ and $\rho_4$ commute, $\{c,d\}\subseteq\mathrm{Fix}(\rho_0)$. 
 By Lemma~\ref{Olivia1} an edge connecting a vertex of $\overline{\mathrm{Fix}(\rho_0)}$ with  a vertex of $\mathrm{Fix}(\rho_0)$ must have label  $1$. 
 As $\mathcal{G}_4$ has a pendant $3$-edge, the vertices of this edge must belong to $\mathrm{Fix}(\rho_0)$. Hence, we have identified the three fixed points of $\rho_0$.
But then as $\mathcal{G}$ is connnected there is a $1$-edge from a vertex of $\mathrm{Fix}(\rho_0)$  and a vertex of $\overline{\mathrm{Fix}(\rho_0)}$. But then by the commuting property the $4$-split belongs to a $\{1,4\}$-square, a contradiction. Therefore $\rho_0$ is a $2$-transposition and, by duality, $\rho_4$ is also $2$-transposition. 

\underline{$\rho_2$ is a $2$-transposition:} If $\rho_2$ is a $4$-transposition then, as the three permutations $\rho_0$, $\rho_2$  and $\rho_4$ commute pairwisely, the vertices of the $0$-split and the vertices of the $4$-split must belong to $\mathrm{Fix}(\rho_2)$. Thus $\{a,b,c,d\}\subseteq \mathrm{Fix}(\rho_2)$, a contradiction.

\underline{$\rho_1$ and $\rho_3$ are $2$-transpositions:} Suppose that $\rho_1$ is a $4$-transposition. By Lemma~\ref{Olivia1} an edge connecting a vertex of $\overline{\mathrm{Fix}(\rho_1)}$ with  a vertex of $\mathrm{Fix}(\rho_1)$ must have either label  $0$ or $2$. Then, as the $4$-split  $\{c,d\}$  does not belong to a square, $\{c,d\}\subseteq \mathrm{Fix}(\rho_1)$. As $\mathcal{G}_4$ has a pendant 3-edge, the vertices of this edge must also belong to $\mathrm{Fix}(\rho_1)$. However with this, there are no possibilities for a 0-split with a trivial orbit, a contradiction. Hence $\rho_1$ is a $2$-transposition and, by duality, $\rho_3$ is also a $2$-transposition.

As $0$ and $4$ are labels of splits, $\mathcal{G}$ has neither $\{0,k\}$-squares ($k\neq 0$) nor $\{4,k\}$-squares  ($k\neq 4$). Consequently $\mathcal{G}$  also does not have  $\{0,4\}$-edges. Thus the $0$-edges and the $4$-edges have no vertices in common. But then there exists a $2$-edge meeting either a $0$-edge or a $4$-edge. This implies that $\mathcal{G}$ either has a $\{0,2\}$-edge or a $\{2,4\}$-edge. Up to duality we may assume that $\mathcal{G}$ has a $\{0,2\}$-edge.
Then $\mathcal{G}_{1,3}$ is the following graph.
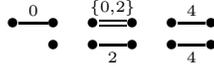
\begin{small}
$$\xymatrix@-1.8pc{ & *{\bullet}\ar@{-}[rr]^0 && *{\bullet}&& *{\bullet}\ar@{=}[rr]^{\{0,2\}} && *{\bullet}&&  *{\bullet}\ar@{-}[rr]^4 && *{\bullet}\\
 & && *{\bullet}&& *{\bullet}\ar@{-}[rr]_2 && *{\bullet}&& *{\bullet}\ar@{-}[rr]_4 && *{\bullet}}$$
 \end{small}
Now as $\mathcal{G}_{1,3}$ has six connected components  and since $\mathcal{G}$ has exactly two $1$-edges and two $3$-edges, then $\mathcal{G}$ is disconnected, a contradiction.

This proves that $\mathcal{G}$ cannot have both a $0$-split and an $(r-1)$-split.
Consequently, by the Propositions~\ref{O2} to \ref{O5},  it may be assumed up to duality that $\mathcal{G}_0$ has a $2$-fracture graph.

 \end{proof}


Let us consider separately the cases $r=4$ and $r=5$. Assume, up to duality, the $\mathcal{G}_0$ has a $2$-fracture graph.\\

\begin{lemma}\label{4G}
 If  $r=4$ then $\mathcal{G}$  has exactly
\begin{enumerate}
 \item one  $\{1,3\}$-square;
 \item four $1$-edges;
 \item one double $\{1,3\}$-edge, if $\rho_3$ is a $4$-transposition;
 \item four $2$-edges;
\end{enumerate}
\end{lemma}
\begin{proof}

 (a) Suppose that $\mathcal{G}_0$ does not have a $\{1,3\}$-square.
 Then, $\mathcal{G}_0$ cannot have $\{1,2\}$-squares either, for otherwise, any edge incident to one of the vertices of the $\{1,2\}$-square must belong to a $\{1,3\}$-square, a contradiction.  
 Similarly $\mathcal{G}_0$ cannot have other squares nor double $\{1,2\}$-edges nor double $\{2,3\}$-edges.
 Thus two incident edges of $\mathcal{G}_0$  must have consecutive labels and the only admissible double edges of $\mathcal{G}_0$  have label-set  $\{1,3\}$.
We have that $\rho_2$ is a $4$-transposition, otherwise we would have  $\{1,3\}$-squares. 
Let us prove that $\rho_1$ is also a $4$-transposition.
 Suppose that $\rho_1$ is a $2$-transposition. As  $\mathcal{G}$ is connected,  $\mathcal{G}_0$ has at least nine edges, hence $\rho_3$ is a $4$-transposition. Then there exists a $1$-edge meeting a $3$-edge, which is only possible if we have a double $\{1,3\}$-edge, contradicting Lemma~\ref{claudio}.
 Thus $\rho_1$ is a $4$-transposition and, by similar arguments,
 we may also conclude that  $\rho_3$ is a $4$-transposition.
 Since there are no $\{1,2\}$-double edges nor $\{2,3\}$-double edges,  there are three  double $\{1,3\}$-edges. Hence $1$ and $3$ are labels of splits, a contradiction.
 Hence, $\mathcal{G}$ contains a $\{1,3\}$-square. 
 
 Let us prove uniqueness. Suppose that there are two  $\{1,3\}$-squares. Recall that the  $0$-split of $\mathcal{G}$ is adjacent to a pendant $1$-edge of  $\mathcal{G}_0$. The existence of two  $\{1,3\}$-squares and a pendant $1$-edge, implies that $\mathcal{G}_0$ has at least five $1$-edges, which is clearly is not possible.

 (b) As $\mathcal{G}_0$ has  a pendant $1$-edge and  $\mathcal{G}$ has a $\{1,3\}$-square, we conclude that  $\mathcal{G}$ has exactly four $1$-edges.
  
 (c) Suppose that $\rho_3$ is a $4$-transposition. Then, as $\rho_1$ is also a $4$-transposition, $\mathcal{G}_0$ has at least $8+8-10=6$ vertices that belong to both a $3$-edge and a $1$-edge. 
 Since $\rho_1$ commutes with $\rho_3$, and that there can only be one  $\{1,3\}$-square, then there must exist exactly one $\{1,3\}$-double edge.
 
  (d) From (a)-(c) we may conclude that the orbits of $\mathcal{G}_{0,2}$ acting on $O_2$ are one of the following
  
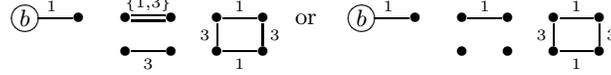
\begin{small}
 $$\xymatrix@-1.7pc{ & *+[o][F]{b} \ar@{-}[rr]^1 && *{\bullet}&& *{\bullet}\ar@{=}[rr]^{\{1,3\}} && *{\bullet}&&  *{\bullet}\ar@{-}[rr]^1 && *{\bullet}\\
 & &&  && *{\bullet}\ar@{-}[rr]_3 && *{\bullet}&& *{\bullet}\ar@{-}[rr]_1\ar@{-}[u]^3 && *{\bullet}\ar@{-}[u]_3
}
\text{ or } \xymatrix@-1.7pc{ & *+[o][F]{b} \ar@{-}[rr]^1 && *{\bullet}&& *{\bullet}\ar@{-}[rr]^1 && *{\bullet}&&  *{\bullet}\ar@{-}[rr]^1 && *{\bullet}\\
 & &&  && *{\bullet}&& *{\bullet}&& *{\bullet}\ar@{-}[rr]_1\ar@{-}[u]^3 && *{\bullet}\ar@{-}[u]_3
}$$
\end{small}
Then, there must be at least three 2-edges connecting the orbits. Hence, $\rho_2$ is a 4-transposition.




\end{proof}

\begin{prop}\label{Gamma01_rank4_orbits}
Let $r=4$. The non-trivial connected components of $\mathcal{G}_{0,1}$ are either as in (1) or as in (2).
 \begin{small}
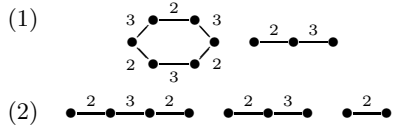
 \[\begin{array}{cc}

 (1) &\xymatrix@-1.8pc{ & *{\bullet}\ar@{-}[rr]^2 && *{\bullet}\ar@{-}[dr]^3  \\
 *{\bullet}\ar@{-}[dr]_2\ar@{-}[ur]^3 & && & *{\bullet}&& *{\bullet}\ar@{-}[rr]^2 && *{\bullet}\ar@{-}[rr]^3 && *{\bullet}\\
 & *{\bullet}\ar@{-}[rr]_3 && *{\bullet}\ar@{-}[ur]_2
 }\\
 (2)& \xymatrix@-1.8pc{ *{\bullet}\ar@{-}[rr]^2 && *{\bullet}\ar@{-}[rr]^3 && *{\bullet}\ar@{-}[rr]^2 && *{\bullet}&& *{\bullet} \ar@{-}[rr]^2 && *{\bullet}\ar@{-}[rr]^3 && *{\bullet}&& *{\bullet}\ar@{-}[rr]^2 && *{\bullet}\\
}
\end{array}\]
\end{small}
\end{prop}
\begin{proof}

 The group $\tilde{G}=\langle \beta_0,\rho_1,\rho_2,\rho_3\rangle$ is a transitive group on  $10$ points and $\tilde{G}_0=G_0$ is transitive. Hence Lemmas~\ref{l0} and \ref{l2} can be used to restrict the sizes of the orbits of $G_{0,1}$.

Let $s$ denote the size of the largest connected component of  $\mathcal{G}_{0,1}$.  Let us consider separately all the possibilities for $s$. Recall that the $0$-split $\{a,b\}$ of $\mathcal{G}$ must be adjacent to a $1$-edge, to be precise, $\mathcal{G}_0$ has a pendant $1$-edge. Furthermore $\mathcal{G}_{0,1}$ has a pendant $2$-edge and fixes the points $a$ and $b$.  Hence $s\notin\{10,11\}$. Moreover, as $\rho_0$ is an even permutation, $s\notin \{8,9\}$. In addition $s>2$ for otherwise $\rho_2$ and $\rho_3$ would commute, a contradiction.
  
 \underline{$s=7$}: In this case the non-trivial components of $\mathcal{G}_{0,1}$ are an alternating path with the sequence of labels $(2,3,2,3,2,3)$ and a double $\{2,3\}$-edge. 
 Since $\rho_3$ is a $4$-transposition, by Proposition~\ref{4G} (a) and (c),  $\mathcal{G}$ has a $\{1,3\}$-square and a double $\{1,3\}$-edge. 
 Moreover the $\{1,3\}$-square must have the double $\{2,3\}$-edge, which implies that $\mathcal{G}$ either has a split with label $3$ or $G_3$ is transitive, a contradiction.
 
 \underline{$s=6$}: The largest orbit is either a path or a hexagon. Suppose first it is a path. Then  $\mathcal{G}_{0,1}$ has three isolated vertices and the following non-trivial components.
 \begin{small}
$$\xymatrix@-1.8pc{ *{\bullet}\ar@{-}[rr]^2 && *{\bullet}\ar@{-}[rr]^3 && *{\bullet}\ar@{-}[rr]^2 && *{\bullet}\ar@{-}[rr]^3 && *{\bullet}\ar@{-}[rr]^2  && *{\bullet} && *{\bullet}\ar@{-}[rr]^2 && *{\bullet}
}$$
\end{small}
Then the $1$-edges of the unique $\{1,3\}$-square of $\mathcal{G}$ are between vertices of the path. But then $\mathcal{G}$ must have at least another three $1$-edges to connect the remaining components (besides the $0$-split), a contradiction.

Now suppose that the largest component is a hexagon. Then, as we need a pendant edge labelled 2, the non-trivial components of $\mathcal{G}_{0,1}$ are as in (1).

\underline{$s=5$}:  In this case  we have the following possibilities for the non-trivial components of $\mathcal{G}_{0,1}$.
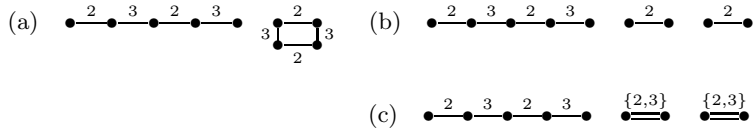
\begin{small}
\[\begin{array}{cc}
\mbox{(a)} \quad \xymatrix@-1.8pc{ *{\bullet}\ar@{-}[rr]^2 && *{\bullet}\ar@{-}[rr]^3 && *{\bullet}\ar@{-}[rr]^2 && *{\bullet}\ar@{-}[rr]^3 && *{\bullet} && *{\bullet}\ar@{-}[rr]^2 && *{\bullet}\\
&& && && && && *{\bullet}\ar@{-}[u]^3 \ar@{-}[rr]_2 && *{\bullet}\ar@{-}[u]_3}
&
\mbox{(b)} \quad \xymatrix@-1.8pc{ *{\bullet}\ar@{-}[rr]^2 && *{\bullet}\ar@{-}[rr]^3 && *{\bullet}\ar@{-}[rr]^2 && *{\bullet}\ar@{-}[rr]^3 && *{\bullet}&&  *{\bullet}\ar@{-}[rr]^2 && *{\bullet}&&  *{\bullet}\ar@{-}[rr]^2 && *{\bullet}}\\
\\[-5pt]
&\mbox{(c)} \quad \xymatrix@-1.8pc{ *{\bullet}\ar@{-}[rr]^2 && *{\bullet}\ar@{-}[rr]^3 && *{\bullet}\ar@{-}[rr]^2 && *{\bullet}\ar@{-}[rr]^3 && *{\bullet}&&  *{\bullet}\ar@{=}[rr]^{\{2,3\}} && *{\bullet}&&  *{\bullet}\ar@{=}[rr]^{\{2,3\}} && *{\bullet}\\
}
\end{array}\]
\end{small}
In (a) there is only one even component of size $4\not\equiv 2 \mod 4$,  a contradiction. In (b), similar to the case when $s=6$, the existence of a $\{1,3\}$-square and the connectness of $\mathcal{G}$ forces the existence of at least five $1$-edges, a contradiction. In (c) the $1$-edges connecting these components must belong to at least two   $\{1,3\}$-squares, a contradiction.

\underline{$s=4$}: The largest component of $\mathcal{G}_{0,1}$ must be either a square or  a path.
Assume first it is a square. If $\mathcal{G}_{0,1}$ has another component of size $4$, then the action $\rho_0$ is odd, a contradiction. Now, as $\mathcal{G}$ has four $2$-edges, $\mathcal{G}_{0,1}$ has exactly three non-trivial components. Hence the possibilities are as follows.
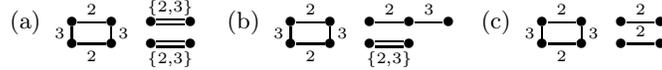
\begin{small}
\[\begin{array}{ccc}
\mbox{(a)} \; \xymatrix@-1.8pc{ *{\bullet}\ar@{-}[rr]^2 && *{\bullet}&& *{\bullet}\ar@{=}[rr]^{\{2,3\}} && *{\bullet}  \\
 *{\bullet}\ar@{-}[u]^3 \ar@{-}[rr]_2 && *{\bullet}\ar@{-}[u]_3 && *{\bullet}\ar@{=}[rr]_{\{2,3\}}&&*{\bullet} 
} 
&
\mbox{(b)}\; \xymatrix@-1.8pc{ *{\bullet}\ar@{-}[rr]^2 && *{\bullet}&& *{\bullet}\ar@{-}[rr]^2 && *{\bullet}\ar@{-}[rr]^3  && *{\bullet}  \\
 *{\bullet}\ar@{-}[u]^3 \ar@{-}[rr]_2 && *{\bullet}\ar@{-}[u]_3 && *{\bullet}\ar@{=}[rr]_{\{2,3\}} &&*{\bullet} 
}
&
\mbox{(c)} \; \xymatrix@-1.8pc{ *{\bullet}\ar@{-}[rr]^2 && *{\bullet}&& *{\bullet}\ar@{-}[rr]^2 && *{\bullet}   \\
 *{\bullet}\ar@{-}[u]^3 \ar@{-}[rr]_2 && *{\bullet}\ar@{-}[u]_3 && *{\bullet}\ar@{-}[rr]^2 &&*{\bullet} 
}
\end{array}\]
\end{small}
In (a) and (b)  the $1$-edges between these components, will induce more than one  $\{1,3\}$-square, a contradiction. In (c) the existence of a  $\{1,3\}$-square implies that $\mathcal{G}$ is disconnected, a contradiction.

Now suppose that the largest component is a path. Since, by Lemma~\ref{4G} (d),  $\mathcal{G}$ has four $2$-edges, the path must have the sequence of labels  $(2,\,3,\,2)$. 
Thus $\mathcal{G}$ has only two $3$-edges, the ones belonging to the $\{1,3\}$-square (that must exist by Lemma~\ref{4G} (a)). 
If there is another path with four vertices, then the action of $\rho_0$ is odd, a contradiction.
Moreover, if there is a $\{2,3\}$-edge, then by Lemma~\ref{claudio} either $\mathcal{G}$ has a $3$-split or $G_3$ is transitive, a contradiction.
Then there is only one possibility corresponding to the graph (2) of this proposition.

\underline{$s=3$}: Lastly, if the largest orbit has three points, and since $\rho_2$ is a 4-transposition, we have one of the following possibilities for the non-trivial components of  $\mathcal{G}_{0,1}$.
\begin{small}
\[\begin{array}{ccccc}
\mbox{(a)}& \xymatrix@-1.7pc{ *{\bullet}\ar@{-}[rr]^2 && *{\bullet}\ar@{-}[rr]^3 && *{\bullet}&*{\bullet}\ar@{=}[rr]^{\{2,3\}} && *{\bullet}& *{\bullet}\ar@{-}[rr]^2 && *{\bullet}& *{\bullet}\ar@{-}[rr]^2 && *{\bullet}}&&
\mbox{(b)} &\xymatrix@-1.7pc{*{\bullet}\ar@{-}[rr]^2 && *{\bullet}\ar@{-}[rr]^3 && *{\bullet}& *{\bullet}\ar@{=}[rr]^{\{2,3\}}  && *{\bullet}& *{\bullet}\ar@{=}[rr]^{\{2,3\}} && *{\bullet}& *{\bullet}\ar@{=}[rr]^{\{2,3\}}  && *{\bullet}}
\end{array}\]
\end{small}
In both cases, by Lemma~\ref{claudio} either there is a $3$-split or $G_3$ is transitive, a contradiction.

Hence, the only possibilities are  the ones stated in this proposition.
\end{proof}

In what follows we analyse the situation when $\Gamma$ has rank five.\\

\begin{lemma}\label{5G}
Let $r=5$. The permutation representation graph  $\mathcal{G}$  has
\begin{enumerate}
 \item exactly two $4$-edges, no double edges with  label $4$ neither  $\{3,4\}$-squares;
 \item exactly two $3$-edges and no double edges having with label $3$;
 \item a $\{2,4\}$-square and  four $2$-edges.
\end{enumerate}
\end{lemma}
\begin{proof}

(a) If there are four $4$-edges a minimal path in $\mathcal{G}$ starting in the vertex $a$ and containing the four $4$-edges must have size at least $4+8=12$ by Lemma~\ref{path}. This gives a contradiction.
By Lemma~\ref{claudio}, $\mathcal{G}$ does not have double $\{i,4\}$-edges ($i\neq 4$). If there is a $\{3,4\}$-square, then $\rho_3$ and $\rho_4$ would commute, a contradiction.
  
(b) Suppose that $\rho_3$ is a $4$-transposition. Consider  first that $\rho_1$ is a $2$-transposition. As $\rho_1$ commutes with $\rho_3$ and $\mathcal{G}_0$  has a pendant $1$-edge, then $\mathcal{G}$ has  a $\{1,3\}$-double edge, but then by Lemma~\ref{claudio} $\mathcal{G}$ has a $1$-split, a contradiction.
Hence, $\rho_1$ is a $4$-transposition. 
  As $\rho_1$ commutes with $\rho_3$ and they move at most $10$ points, then there are at least $6=8+8-10$ vertices moved by both $\rho_1$ and $\rho_3$. If  there are no $\{1,3\}$-squares, there are three $\{1,3\}$-double edges, but then $\mathcal{G}$ has a $1$-split and a $3$-split, a contradiction. This shows that $\mathcal{G}$ has a $\{1,3\}$-square. As $\mathcal{G}_0$ has a pendant $1$-edge, there also exists exactly one $\{1,3\}$-double edge.  This determines the graph $\mathcal{G}_{0,2,4}$ which has exactly five components: a $\{1,3\}$-square, a $\{1,3\}$-double edge, a $1$-edge, a $3$-edge and a single vertex. Now by (a)  the two $4$-edges must connect vertices in different components of  $\mathcal{G}_{0,2,4}$.
  This gives the following  possibilities for the non-trivial components of $\mathcal{G}_{0,2}$.
  \begin{small}
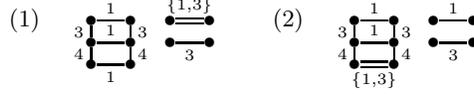
\[\begin{array}{cccccc}
\mbox{(1)}&\xymatrix@-1.8pc{*{\bullet}\ar@{-}[rr]^1 && *{\bullet}&& *{\bullet}\ar@{=}[rr]^{\{1,3\}} && *{\bullet}\\
  *{\bullet}\ar@{-}[rr]^1\ar@{-}[u]^3 && *{\bullet}\ar@{-}[u]_3  && *{\bullet}\ar@{-}[rr]_3 && *{\bullet}\\
  *{\bullet}\ar@{-}[rr]_1\ar@{-}[u]^4 && *{\bullet}\ar@{-}[u]_4}&&
\mbox{(2)} & \xymatrix@-1.8pc{*{\bullet}\ar@{-}[rr]^1 && *{\bullet}&& *{\bullet}\ar@{-}[rr]^1 && *{\bullet}\\
  *{\bullet}\ar@{-}[rr]^1\ar@{-}[u]^3 && *{\bullet}\ar@{-}[u]_3  && *{\bullet}\ar@{-}[rr]_3 && *{\bullet}\\
  *{\bullet}\ar@{=}[rr]_{\{1,3\}} \ar@{-}[u]^4 && *{\bullet}\ar@{-}[u]_4}
\end{array}\]
\end{small}

In (1) $\mathcal{G}_0$ does not have a pendant $1$-edge, and in (2) there exists a $1$-split. In both cases we get a contradiction.
This shows that $\rho_3$ is a $2$-transposition. By Lemma~\ref{claudio}, the label $3$ cannot be one of the labels of a double edge.

(c) Suppose that there is no  $\{2,4\}$-square. Let us first deal with the case where $\mathcal{G}$ has a $\{1,4\}$-square.
Since the $\rho_3$ and $\rho_4$ are $2$-transpositions and cannot commute with each other, then the non-trivial component of $\mathcal{G}_0$ is one of the following two graphs.
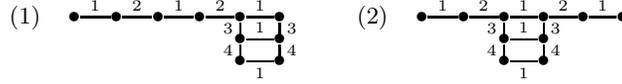
\begin{small}
 \[\begin{array}{ccccc}
\mbox{(1)} &\xymatrix@-1.8pc{*{\bullet}\ar@{-}[rr]^1 && *{\bullet}\ar@{-}[rr]^2 && *{\bullet}\ar@{-}[rr]^1 && *{\bullet}\ar@{-}[rr]^2 && *{\bullet}\ar@{-}[rr]^1 && *{\bullet}\\
  && && && &&*{\bullet}\ar@{-}[rr]^1\ar@{-}[u]^3 && *{\bullet}\ar@{-}[u]_3\\
  && && && &&  *{\bullet}\ar@{-}[rr]_1\ar@{-}[u]^4 && *{\bullet}\ar@{-}[u]_4}&&
 \mbox{(2)} & \xymatrix@-1.8pc{*{\bullet}\ar@{-}[rr]^1 && *{\bullet}\ar@{-}[rr]^2 && *{\bullet}\ar@{-}[rr]^1 && *{\bullet}\ar@{-}[rr]^2 && *{\bullet}\ar@{-}[rr]^1 &&  *{\bullet}\\
  && && *{\bullet}\ar@{-}[rr]^1\ar@{-}[u]^3 && *{\bullet}\ar@{-}[u]_3\\
  && &&  *{\bullet}\ar@{-}[rr]_1\ar@{-}[u]^4 && *{\bullet}\ar@{-}[u]_4}\\
  \end{array}\]
  \end{small}
 In both cases $\rho_1$ is odd, a contradiction.
 
 Now consider that  $\mathcal{G}$ does not have $\{i,4\}$-squares for $i\in\{1,2,3\}$. 
 The   $\{1,3\}$-squares are also forbidden otherwise there must exist a $4$-edge incident to a vertex of this square (recall that $\rho_3$ is a $2$-transposition and $\rho_3$ cannot commute with $\rho_4$), but then $\mathcal{G}$ has an
   $\{1,4\}$-square, a contradiction. By a similar argument $\mathcal{G}$ does not have  $\{2,3\}$-squares. Keep in mind that by (a) and (b),  $\mathcal{G}$ does not have double edges containing the labels $3$ or $4$.
 Hence $\mathcal{G}_0$ restricted to $O_2$ must be one of the following graphs.

 \begin{small}
   $$\mbox{(1)}\; \xymatrix@-1.8pc{*{\bullet}\ar@{-}[rr]^2 && *{\bullet}\ar@{-}[rr]^1 && *{\bullet}\ar@{-}[rr]^2 && *{\bullet}\ar@{-}[rr]^1 && *{\bullet}\ar@{-}[rr]^2  && *{\bullet}\ar@{-}[rr]^3 &&  *{\bullet}\ar@{-}[rr]^4  && *{\bullet}\ar@{-}[rr]^3  && *{\bullet}\ar@{-}[rr]^4  && *{\bullet}}\quad \mbox{(2)}\;
  \xymatrix@-1.8pc{ *{\bullet}\ar@{-}[rr]^4 && *{\bullet}\ar@{-}[rr]^3 && *{\bullet}\ar@{-}[rr]^2 && *{\bullet}\ar@{-}[rr]^1 && *{\bullet}\ar@{-}[rr]^2  && *{\bullet}\ar@{-}[rr]^1 &&  *{\bullet}\ar@{-}[rr]^2   && *{\bullet}\ar@{-}[rr]^3  && *{\bullet}\ar@{-}[rr]^4  && *{\bullet} }$$
   \end{small}

   But in these cases neither $\mathcal{G}_0$ has a pendant $1$-edge, nor is $\rho_2$ an even permutation, a contradiction. This proves the existence of a $\{2,4\}$-square. 
   
 Now suppose  that $\rho_2$ is a $2$-transposition. Since $\rho_1$ cannot commute with $\rho_2$, there must exist a $1$-edge incident to a vertex of the  $\{2,4\}$-square but then there exists another  $\{1,4\}$-square, making $\rho_4$ is a $4$-transposition, contradicting (a).

Consequently $\rho_2$ is a $4$-transposition.

\end{proof}

\begin{prop}\label{Gamma01_rank5_orbits}
Let $r=5$. The non-trivial components of $\mathcal{G}_{0,1}$ are either as in (a), (b) or (c).
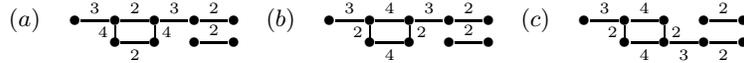
\begin{small}
 \[\begin{array}{cccccc}
(a) & \xymatrix@-1.8pc{ *{\bullet}\ar@{-}[rr]^3 && *{\bullet}\ar@{-}[rr]^2 && *{\bullet}\ar@{-}[rr]^3 && *{\bullet}\ar@{-}[rr]^2 && *{\bullet}\\
 && *{\bullet}\ar@{-}[rr]_2\ar@{-}[u]^4 && *{\bullet}\ar@{-}[u]_4 && *{\bullet}\ar@{-}[rr]^2 && *{\bullet}
 }
 &
(b)&\xymatrix@-1.8pc{ *{\bullet}\ar@{-}[rr]^3 && *{\bullet}\ar@{-}[rr]^4 && *{\bullet}\ar@{-}[rr]^3 && *{\bullet}\ar@{-}[rr]^2 && *{\bullet}\\
 && *{\bullet}\ar@{-}[rr]_4\ar@{-}[u]^2 && *{\bullet}\ar@{-}[u]_2 && *{\bullet}\ar@{-}[rr]^2 && *{\bullet}
 }
 &
 (c)&
 \xymatrix@-1.8pc{ *{\bullet}\ar@{-}[rr]^3 && *{\bullet}\ar@{-}[rr]^4 && *{\bullet}&& *{\bullet}\ar@{-}[rr]^2 && *{\bullet}\\
 && *{\bullet}\ar@{-}[rr]_4\ar@{-}[u]^2 && *{\bullet}\ar@{-}[u]_2 \ar@{-}[rr]_3 && *{\bullet}\ar@{-}[rr]_2 && *{\bullet}
 }                                                                                                                                                                                       
  \end{array}\]
 \end{small}
\end{prop}
\begin{proof}
  
The group $\tilde{G}=\langle \beta_0,\rho_1,\rho_2,\rho_3,\rho_4\rangle$ is a transitive group on $10$ points and $\tilde{G}_0=G_0$ is transitive. Hence Lemmas~\ref{l0} and \ref{l2} can be used to restrict the sizes of the orbits of $G_{0,1}$. 
 Let $s$ be the size of the largest orbit of $\mathcal{G}_{0,1}$. Recall that $a$ and $b$ (the vertices of the $0$-split) are isolated vertices of $\mathcal{G}_{0,1}$. Hence $s\leq 9$.
 
 \underline{$s=4$}: In this case the  $\{2,4\}$-square (that exists by Proposition~\ref{5G} (c)) determines a maximal component of $\mathcal{G}_{0,1}$. As, by Proposition~\ref{5G} (a), $\rho_4$ is a $2$-transposition and the $\{2,4\}$-square must be connected to the rest of the permutation graph, this can only happen via a $3$-edge. But then $s>4$, a contradiction.
  
 \underline{$s=5$}:  By what we have proved in the previous case, we have the following possibilities for the largest orbit of size 5.
 
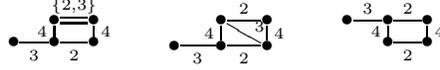
\begin{small}
 $$\xymatrix@-1.8pc{&&*{\bullet}\ar@{=}[rr]^{\{2,3\}} && *{\bullet}\\ *{\bullet}\ar@{-}[rr]_3&&*{\bullet}\ar@{-}[rr]_2 \ar@{-}[u]^4&& *{\bullet}\ar@{-}[u]_4 } \quad\quad \xymatrix@-1.7pc{&&*{\bullet}\ar@{-}[rr]^2\ar@{-}[rrd]^(0.7)3 && *{\bullet}\\ *{\bullet}\ar@{-}[rr]_3&&*{\bullet}\ar@{-}[rr]_2 \ar@{-}[u]^4&& *{\bullet}\ar@{-}[u]_4 } \quad\quad 
 \xymatrix@-1.8pc{*{\bullet}\ar@{-}[rr]^3&&*{\bullet}\ar@{-}[rr]^2 && *{\bullet}\\ && *{\bullet}\ar@{-}[rr]_2 \ar@{-}[u]^4&& *{\bullet}\ar@{-}[u]_4 }$$ 
 \end{small}
 In the first case, $\mathcal{G}$ has a $\{2,3\}$-edge, contradicting Lemma~\ref{5G} (b). In the second graph $G_4$ is transitive, a contradiction.  Hence, we may assume the largest orbit of size $5$ is given by the third graph. 
 Since we cannot have an orbit with a $\{2,3\}$-double edge (by Lemma~\ref{5G} (b)) and we need a pendant edge with label $2$, then the non-trivial components of $\mathcal{G}_{0,1}$ are as follows.
 
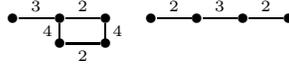
\begin{small}
 $$\xymatrix@-1.7pc{*{\bullet}\ar@{-}[rr]^3&&*{\bullet}\ar@{-}[rr]^2 && *{\bullet}&& *{\bullet}\ar@{-}[rr]^2&&*{\bullet}\ar@{-}[rr]^3 && *{\bullet}\ar@{-}[rr]^2&& *{\bullet}\\ && *{\bullet}\ar@{-}[rr]_2 \ar@{-}[u]^4&& *{\bullet}\ar@{-}[u]_4 }$$
 \end{small}
 
By similar arguments to the ones given by Lemmas~\ref{l1} and \ref{l2}, we  get  that $\mathcal{G}$ has an odd number of $0$-edges, contradiction. 

  \underline{$s=6$}: In this case the largest component must contain the $\{2,4\}$-square.  If the $3$-edges are both incident to the $\{2,4\}$-square then we get the following possibilities for that component of size 6 (keep in mind that $\rho_2$ is a 4-transposition).
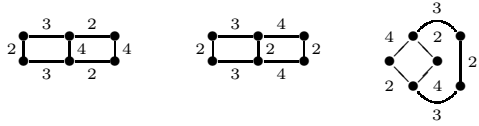
\begin{small}   
   $$\xymatrix@-1.7pc{*{\bullet}\ar@{-}[rr]^3 &&*{\bullet}\ar@{-}[rr]^2 && *{\bullet}\\ *{\bullet}\ar@{-}[rr]_3\ar@{-}[u]^2 &&*{\bullet}\ar@{-}[rr]_2 \ar@{-}[u]_4 && *{\bullet}\ar@{-}[u]_4}\quad\quad \xymatrix@-1.7pc{*{\bullet}\ar@{-}[rr]^3 &&*{\bullet}\ar@{-}[rr]^4 && *{\bullet}\\ *{\bullet}\ar@{-}[rr]_3\ar@{-}[u]^2 &&*{\bullet}\ar@{-}[rr]_4 \ar@{-}[u]_2 && *{\bullet}\ar@{-}[u]_2}  \quad\quad \xymatrix@-1.7pc{
  & *{\bullet}\ar@{-}[dr]^2 \ar@/^/@{-}[rr]^3 && *{\bullet}\ar@{-}[dd]^2\\
  *{\bullet}\ar@{-}[ur]^4  \ar@{-}[dr]_2 && *{\bullet}\\
   & *{\bullet}\ar@{-}[ur]_4\ar@/_/@{-}[rr]_3 && *{\bullet}}$$
   \end{small}
  But then $G_4$ is transitive  or, by Lemma~\ref{5G} (b), $\rho_2$ and $\rho_3$ commute. In any case we get a contradiction.
   
Now suppose that there exists  only one $3$-edge incident to $\{2,4\}$-square.  Having in mind that $\mathcal{G}_0$ does not have fracture graph and does not have $\{2,3\}$-edges (by Lemma~\ref{5G}), then the largest component of $\mathcal{G}_{0,1}$ is as follows.
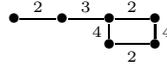
\begin{small}
$$\xymatrix@-1.7pc{*{\bullet}\ar@{-}[rr]^2&&*{\bullet}\ar@{-}[rr]^3&&*{\bullet}\ar@{-}[rr]^2 && *{\bullet}\\ && && *{\bullet}\ar@{-}[rr]_2 \ar@{-}[u]^4&& *{\bullet}\ar@{-}[u]_4 }$$
\end{small}

But now, as by Lemma~\ref{5G} (b) $\mathcal{G}$ does not have double $\{2,3\}$-edges and $G_4$ is not transitive, the non-trivial orbits of $\mathcal{G}_{0,1}$ are as follows.
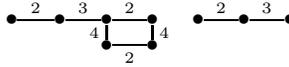
\begin{small}
 $$\xymatrix@-1.7pc{*{\bullet}\ar@{-}[rr]^2&&*{\bullet}\ar@{-}[rr]^3&&*{\bullet}\ar@{-}[rr]^2 && *{\bullet}&& *{\bullet}\ar@{-}[rr]^2&& *{\bullet}\ar@{-}[rr]^3 && *{\bullet}\\ && && *{\bullet}\ar@{-}[rr]_2 \ar@{-}[u]^4&& *{\bullet}\ar@{-}[u]_4 *}$$
 \end{small}
 If $\rho_0$ acts non-trivially on these orbits, then it can only happen on the largest orbit. However, this implies that $\mathcal{G}$ has a double  $\{0,3\}$-edge, a contradiction by Lemma~\ref{5G}(b). Hence, $\rho_0$ is odd, a contradiction.
   
  \underline{$s=7$}: The largest orbit with size 7 can be either the ones presented in the statement of this proposition or the following. 
  
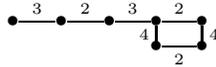
\begin{small}
 $$\xymatrix@-1.7pc{*{\bullet}\ar@{-}[rr]^3 &&*{\bullet}\ar@{-}[rr]^2&&*{\bullet}\ar@{-}[rr]^3&&*{\bullet}\ar@{-}[rr]^2 && *{\bullet}\\ && &&  && *{\bullet}\ar@{-}[rr]_2 \ar@{-}[u]^4&& *{\bullet}\ar@{-}[u]_4 }$$ 
 \end{small}
 In any case the other non-trivial component of   $\mathcal{G}_{0,1}$ is a single $2$-edge. In the case above it is not possible to connect the two components of  $\mathcal{G}_{0,1}$ with $1$-edges (recall that the graph has exactly two $4$-edges and two $3$-edges by Lemma~\ref{5G}).
 
  \underline{$s\in\{8,9\}$}: In this case $\rho_0$ is odd.
 
The only possibilities for  $\mathcal{G}_{0,1}$ are the ones stated in the proposition.
\end{proof}


\begin{prop}\label{isplit}
 Let $G$ be an even  group of degree $11$ and rank $r\leq 5$.
 If $\Gamma$ has a fracture graph then either $\Gamma$ has a $2$-fracture graph or $\mathcal{G}$ is as follows.
$$ 
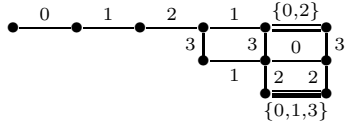
\xymatrix@-1.5pc{  *{\bullet} \ar@{-}[rr]^0 &&*{\bullet} \ar@{-}[rr]^1&&*{\bullet}  \ar@{-}[rr]^2 &&*{\bullet}\ar@{-}[d]_3  \ar@{-}[rr]^1&&*{\bullet} \ar@{=}[rr]^{\{0,2\}}\ar@{-}[d]_3&&*{\bullet}\ar@{-}[d]^3\\
 &&&& &&*{\bullet} \ar@{-}[rr]_1&&*{\bullet} \ar@{-}[rr]^0\ar@{-}[d]^2&&*{\bullet}\ar@{-}[d]_2\\
 &&&&&&&&*{\bullet}  \ar@{-}[rr]\ar@{=}[rr]_{\{0,1,3\}}&&*{\bullet}}$$
\end{prop}

\begin{proof}
If $\Gamma$ does have a split then either $0$ or $r-1$ is a label of a split but not both by Propositions~\ref{all} and \ref{not2split}.
Suppose without loss of generality that  $\mathcal{G}_0$ has a $2$-fracture graph. By Propositions~\ref{4G} and~\ref{Gamma01_rank4_orbits} the only possibilities when the rank of $\Gamma$ is equal to $4$ are  the sggi's (E1)-(E10) of the appendix  which are not string C-groups, or  $\Gamma$ has the permutation representation graph given in the statement of this theorem. 

If $r=5$, then by Propositions~\ref{5G} and~\ref{Gamma01_rank5_orbits} we find the possibilities  (E11), (E12) and (E13) of  the appendix which again are not string C-groups, by Proposition~\ref{notC}, a contradiction.
 \end{proof}

Later we will see that the permutation representation graph given in the above proposition corresponds to the $11$-cell.

\section{When $\mathcal{G}$ has a $2$-fracture graph}\label{2fracture11}


In this section we assume that $\Gamma=(G, \{\rho_0,\ldots, \rho_r\})$ with $r\in\{4,5\}$, is a string C-group representation for an even transitive group of degree $11$ having a permutation representation graph $\mathcal{G}$ that admits a $2$-fracture graph.

In Proposition 4.9 of  \cite{2023CFL} the authors give a classification of the  string C-groups of degree $n$ admitting a $2$-fracture of rank at least  $(n-1)/2$. 
If $\Gamma$ has rank $5$, then a $2$-fracture graph has exactly $10$ edges and $11$ vertices. Then if it is connected it is a tree, otherwise there is a $2$-fracture graph having exactly two components, one being a tree and the other one having an alternating square \cite[Proposition 4.12]{2017CFLM}. Having this in mind it is possible to find all possibilities for $\mathcal{G}$. This was precisely the idea behind the classification given in Proposition 4.9 of \cite{2023CFL}.
A consequence of this is the following.

\begin{prop}\label{2frac5}
If $\mathcal{G}$ has a $2$-fracture graph, then $r\neq 5$.
\end{prop}

We now consider the case $r=4$. 

\begin{lemma}\label{lem:2fractaltsq}
Suppose that $\rho_i$ is a $4$-transposition. 
If $\rho_j$ swaps a pair of vertices of  $\mathrm{Fix}(\rho_i)$  and $| j - i| \neq 1$ then $\mathcal{G}$ has an $\{i,j\}$-edge and an $\{i,j\}$-square. In particular, $\rho_j$ is a $4$-transposition.
\end{lemma}
\begin{proof}
In this case $|\mathrm{Fix}(\rho_i)|=3$.
As $\rho_j$ is even, by Lemma~\ref{Olivia1}, $\rho_j$ must swap an odd number of pairs of vertices in $\overline{\mathrm{Fix}(\rho_i)}$.
The rest follows from the fact that $\mathcal{G}$ has a $2$-fracture graph.
\end{proof}

\begin{prop}\label{2fract}
If $\mathcal{G}$ has a $2$-fracture graph, then $r\neq 4$
\end{prop}

\begin{proof}
The connectedness of $\mathcal{G}$ implies that  the generating set of $G$ contains at least one permutation which is a $4$-transposition, suppose first it is $\rho_0$. 
Let $\mathrm{Fix}(\rho_0) = \{ A,B,C \}$.
In this case, $\rho_1$ is the only generator which may interchange  a vertex of $\mathrm{Fix}(\rho_0)$ with  a vertex of $\overline{\mathrm{Fix}(\rho_0)}$.
Hence up to a relabelling of the points that are fixed by $\rho_0$, we have three cases: $\{A,B\}$ is a $3$-edge (Case 1); $\{A,B\}$ is a $2$-edge (Case 2); the vertices $A$, $B$ and $C$ are fixed by $\rho_2$ and $\rho_3$ (Case 3).
Let us deal with each case separately.

\underline{Case 1}:
By Lemma~\ref{lem:2fractaltsq} $\mathcal{G}$ has  $\{0,3\}$-square, a $\{0,3\}$-edge and another $0$-edge that we denote by $\{D,F\}$. 
Thus there exists a $2$-fracture graph of $\mathcal{G}$ containing the  $\{0,3\}$-square.
The $\{0,3\}$-edge must (for transitivity) be connected to other vertices of the graph. 
It cannot be at distance one from the $\{0,3\}$-square for otherwise the edges of the $\{0,3\}$-square would not belong to a $2$-fracture graph.
If it is at distance one from the other $0$-edge $\{D,F\}$ then it must be via a $2$-edge. Then we get $\{0,2\}$-square (with an extra $3$-edge) and a $\{0,3\}$-square that cannot be at distance one from each other. Thus both squares must be at distance one (via a $1$-edge) from one of the vertices of the set $\{A,B,C\}$ as shown in the following graph, on the left, where the dashed line can only be crossed by $1$-edges. 
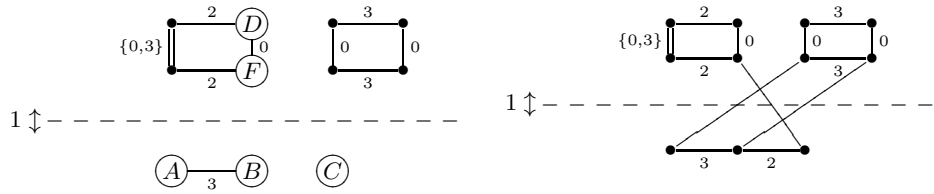
\begin{small}
\[\begin{array}{cc}
\xymatrix@-1.5pc{&& &&  *{\bullet}\ar@{=}[d]_{\{0,3\}} \ar@{-}[rr]^2&& *+[o][F]{D} \ar@{-}[d]^0&&*{\bullet}  \ar@{-}[d]^0  \ar@{-}[rr]^3 &&  *{\bullet} \ar@{-}[d]^0&&\\
                                  &&&& *{\bullet} \ar@{-}[rr]_2 && *+[o][F]{F}  &&*{\bullet}\ar@{-}[rr]_3    &&*{\bullet}&& \\
1\updownarrow \ar@{--}[rrrrrrrrrrrr]&&&&&&&&&&&&\\
&&&& *+[o][F]{A} \ar@{-}[rr]_3&& *+[o][F]{B} && *+[o][F]{C}&&&& } 
&
\xymatrix@-1.4pc{&& &&  *{\bullet}\ar@{=}[d]_{\{0,3\}}   \ar@{-}[rr]^2&& *{\bullet}\ar@{-}[d]^0&&*{\bullet}  \ar@{-}[d]^0  \ar@{-}[rr]^3 &&  *{\bullet} \ar@{-}[d]^0&&\\
                                  &&&& *{\bullet} \ar@{-}[rr]_2 && *{\bullet} &&*{\bullet}\ar@{-}[rr]_3    &&*{\bullet}&& \\
1\updownarrow \ar@{--}[rrrrrrrrrrrr]&&&&&&&&&&&&\\
&&&& *{\bullet}\ar@{-}[rr]_3\ar@{-}[uurrrr]&& *{\bullet}\ar@{-}[rr]_2\ar@{-}[uurrrr]&&*{\bullet}\ar@{-}[uull] &&&& } \end{array}\]
\end{small}
Now $C$ cannot be at distance one from $\{0,3\}$-square thus, for connectedness we get the graph on the right and no more edges can be added for otherwise $\mathcal{G}$ does not have a $2$-fracture graph. Hence $\rho_1$ and $\rho_2$ are odd, a contradiction.
Hence the double $\{0,3\}$-edge cannot be part of a $\{0,2\}$-square and only $A$ and $B$ are at distance one from the $\{0,3\}$-edge.

This forces the existence of $1$-edges from the $\{0,3\}$-edge  to vertices of  $\{A,B,C\}$. 
That is only possible if $\mathcal{G}$ has a $\{1,3\}$-square containing the edge $\{A,B\}$.
Now  connectedness implies $\mathcal{G}$ contains the following graph where the edges that cross the dashed lines have label $1$. 
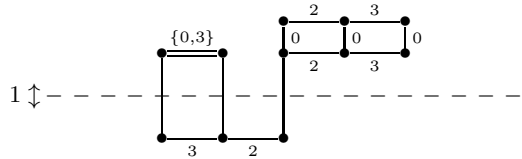
\begin{small}
$$ \xymatrix@-1.5pc{&&&&   &&&&*{\bullet} \ar@{-}[rr]^2\ar@{-}[d]^0 && *{\bullet}  \ar@{-}[d]^0  \ar@{-}[rr]^3 &&  *{\bullet} \ar@{-}[d]^0&&&&\\
                                  &&&& *{\bullet} \ar@{=}[rr]^{\{0,3\}}\ar@{-}[dd]  && *{\bullet}\ar@{-}[dd]  &&*{\bullet}\ar@{-}[rr]_2  \ar@{-}[dd]   &&*{\bullet}\ar@{-}[rr]_3 &&*{\bullet}&&&& \\
1\updownarrow \ar@{--}[rrrrrrrrrrrrrrrr]&&&&&&&&&&&&&&&&\\
&&&& *{\bullet}\ar@{-}[rr]_3&& *{\bullet}\ar@{-}[rr]_2 && *{\bullet}&&&&&& &&} $$
\end{small}

 As $G$ is even, there must be another single $1$-edge, forming a $\{1,3\}$-edge, but then $G_3$ is transitive, giving a contradiction with the fact that $\Gamma$ has a fracture graph.

\underline{Case 2}: Let $\{A,B\}$ be a $2$-edge of $\mathcal{G}$.
By Lemma \ref{lem:2fractaltsq} $\mathcal{G}$ has a $\{0,2\}$-square and  a  $\{0,2\}$-edge. Moreover the $\{0,2\}$-square belongs to a $2$-fracture graph.
As Case 1 gives a contradiction we may assume that  $A,\,B,\,C\in \mathrm{Fix}(\rho_3)$.  As $\mathcal{G}$ has a 2-fracture graph, $\mathcal{G}$ must have exactly two $3$-edges. This gives the following two possibilities for $\mathcal{G}_{1}$.

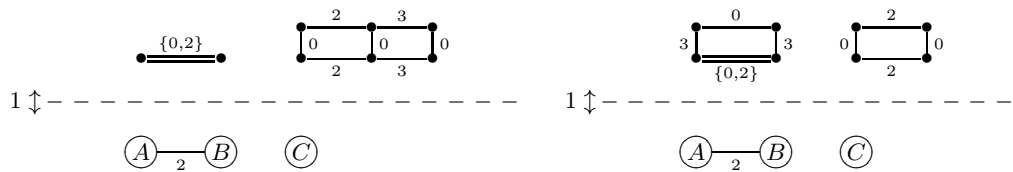
\begin{small}
\[\begin{array}{cc}
\xymatrix@-1.5pc{&&&   &&&&*{\bullet} \ar@{-}[rr]^2\ar@{-}[d]^0 && *{\bullet}  \ar@{-}[d]^0  \ar@{-}[rr]^3 &&  *{\bullet} \ar@{-}[d]^0&&&\\
                                  &&& *{\bullet} \ar@{=}[rr]^{\{0,2\}}&& *{\bullet} &&*{\bullet}\ar@{-}[rr]_2 &&*{\bullet}\ar@{-}[rr]_3 &&*{\bullet}&&& \\
1\updownarrow \ar@{--}[rrrrrrrrrrrrrr]&&&&&&&&&&&&&&\\
&&& *+[o][F]{A} \ar@{-}[rr]_2&& *+[o][F]{B} && *+[o][F]{C}&&&&&& &} &

\xymatrix@-1.5pc{& && *{\bullet} \ar@{-}[rr]^0\ar@{-}[d]_3 && *{\bullet}  \ar@{-}[d]^3   &&*{\bullet}\ar@{-}[rr]^2 \ar@{-}[d]_0&&*{\bullet} \ar@{-}[d]^0&&&\\
                                  &&& *{\bullet} \ar@{=}[rr]_{\{0,2\}}&& *{\bullet}  &&*{\bullet}\ar@{-}[rr]_2&&*{\bullet}&&& \\
1\updownarrow \ar@{--}[rrrrrrrrrrrr]&&&&&&&&&&&&\\
&&& *+[o][F]{A} \ar@{-}[rr]_2&& *+[o][F]{B} && *+[o][F]{C}&&&&& } 

\end{array}\]
\end{small}

Transitivity implies that $\mathcal{G}$ has at least four $1$-edges (three of those connecting the graph). In addition the $1$-edges connecting different components cannot be adjacent to the $3$-edges. Then, in the case on the left, evenness implies the existence of a $\{1,3\}$-edge, making $G_3$ transitive. In the case on the right, the $\{0,3\}$-square cannot be connected to the rest of the graph. Each case leads to a contradiction.

\underline{Case 3}: In this case $A$, $B$ and $C$ are vertices of degree one in $\mathcal{G}$. Thus  $\mathcal{G}$ has at least three $1$-edges (crossing the dotted line in the graph below).
As $\rho_3$ commutes with $\rho_1$, the $3$-edges cannot intersect the $1$-edges which cross the dotted line, thus $\rho_3$ is a $2$-transposition.
Hence we get the following subgraph of $\mathcal{G}$. 
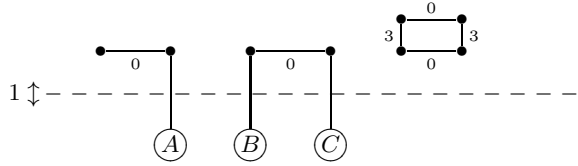
\begin{small}
$$ \xymatrix@-1.5pc{&&   &&&& &&&& *{\bullet}  \ar@{-}[d]_3  \ar@{-}[rr]^0&&  *{\bullet} \ar@{-}[d]^3\\
                                  && *{\bullet} \ar@{-}[rr]_0  && *{\bullet}\ar@{-}[dd]  &&*{\bullet}\ar@{-}[rr]_0  \ar@{-}[dd]   &&*{\bullet}\ar@{-}[dd]&&*{\bullet} \ar@{-}[rr]_0 &&  *{\bullet}&& \\
1\updownarrow \ar@{--}[rrrrrrrrrrrrrrrr]&&&&&&&&&&&&&&&&\\
&&&& *+[o][F]{A} && *+[o][F]{B}&& *+[o][F]{C}&&&&&& } $$
\end{small}
But the existence of a fourth $1$-edge, implies that  $\mathcal{G}$ has a $\{1,3\}$-edge, which by Lemma~\ref{claudio} and the fact the $\mathcal{G}$ has a 2-fracture graph leads to a contradiction.

Hence, the above cases contradict that $\rho_0$ is a $4$-transpositions. Thus $\rho_0$ is a $2$-transposition, and by duality  $\rho_3$ is also a $2$-transposition.
If $\rho_2$ is a $2$-transposition, then, for transitivity, $\rho_1$ must be a $4$-transposition.
In this case $\mathcal{G}$ has exactly ten edges, thus it is a tree.
Then $|\mathrm{Fix}(\rho_1)|=3$ and only $\rho_0$ and $\rho_2$ can swap vertices between $\mathrm{Fix}(\rho_1)$ and $\overline{\mathrm{Fix}(\rho_1)}$. But in order to avoid squares and double edges, $\rho_3$ fixes $\overline{\mathrm{Fix}(\rho_1)}$ pointwisely, thus $\rho_3$  cannot be a $2$-transposition, a contradiction.
Hence $\rho_1$ and $\rho_2$ are $4$-transpositions.
Therefore $\rho_1$ and $\rho_3$ must both act non-trivially on at least one point ($8+4-11=1$). 
Thus by Lemma~\ref{claudio} $\mathcal{G}$ has a $\{1,3\}$-square.
Similarly  $\mathcal{G}$ has a $\{0,2\}$-square.
Let $d$ be the distance between these two squares.
We have that $d\in\{2,4\}$.
If $d=2$ we get the first nine graphs of (F1) to (F9) of the appendix.
If $d=4$ we get only one possibility corresponding to graph (F10) of the appendix.
In any case $\Gamma$ is not a string C-group by Proposition~\ref{notC}, a contradiction.
\end{proof}



\section{A classification of even transitive string C-groups of degree $11$}\label{final}

The only transitive even groups of degree $11$ that we need to consider are $\PSL_2(11)$, $M_{11}$ and $A_{11}$, as those are the unique ones that can be generated by involutions.

\begin{prop}\label{4PSL}
If $r\in\{4,5\}$ and $G\cong  \PSL_2(11)$ then $\Gamma$ is the  abstract regular $4$-polytope known as the $11$-cell and has the permutation representation graph given in Theorem~\ref{isplit}.
\end{prop}

\begin{proof}
The proper even transitive subgroups of  $\PSL_2(11)$ are not generated by involutions \cite{ATLAS}.
Thus if $G \cong \PSL_2(11)$ then $\mathcal{G}$ has a fracture graph. Moreover, by Propositions~\ref{2frac5} and \ref{2fract},  $\mathcal{G}$  has a split.

The $11$-cell is a well known $4$-polytope whose automorphism group is $\PSL_2(11)$ \cite{11cell}.
Consequently, the graph given in Proposition~\ref{isplit} must  be the  permutation representation graph of the $11$-cell.
\end{proof}

From Propositions~\ref{isplit}, ~\ref{2frac5}, \ref{2fract} and ~\ref{4PSL} we have the following.

\begin{coro}\label{casefrac}
If $G$ is an even transitive string C-group of degree $11$ with a fracture graph of rank $r\in\{3,4,5\}$ then one of the following two situations occurs:
\begin{itemize}
\item $\Gamma$ is the  abstract regular $4$-polytope known as the $11$-cell;
\item $\Gamma$ has rank $3$ and has a $2$-fracture graph.
\end{itemize}
\end{coro}

The classification of abstract regular polyhedra (reflexible maps) for $\PSL_2(11)$ can be found in various atlases that are available online, either related to  polytopes or to maps, particularly in \cite{Hatlas} and \cite{LVatlas}.

\begin{lemma}\label{PSLpoly}
There are, up to duality, three abstract regular polyhedra for $PSL_2(11)$. 
\end{lemma}

Let us now determine the faithful transitive permutation representations graphs of the abstract regular polyhedra for $PSL_2(11)$.

\begin{prop}\label{PSL}
There are, up to duality, four abstract regular polytopes for $\PSL_2(11)$ with rank $r\in\{3,4\}$.  Their (five) faithful transitive permutation representation graphs are, up to duality, given in the following table.
\[\begin{tabular}{|cc|}
\hline
Type & Permutation representation graphs \\
\hline
 $\{5,5\}$ & $ \xymatrix@-1.5pc{  &&*{\bullet} \ar@{-}[rr]^1&&*{\bullet}  \ar@{-}[rr]^0   \ar@{-}[d]_2&&*{\bullet} \ar@{-}[rr]^1 \ar@{-}[d]_2&&*{\bullet} \ar@{-}[rr]^0&&*{\bullet} \\
 *{\bullet} \ar@{=}[rr]_{\{0,2\}} &&*{\bullet} \ar@{-}[rr]_1&&*{\bullet}  \ar@{-}[rr]_0 &&*{\bullet} \ar@{-}[rr]_1&&*{\bullet} \ar@{-}[rr]_2&&*{\bullet}}$ \qquad (1)\\
 \hline
  $\{5,6\}$ & $\xymatrix@-1.5pc{  &&*{\bullet} \ar@{-}[d]_2 \ar@{-}[rr]^0&&*{\bullet}  \ar@{-}[rr]^1   \ar@{-}[d]_2&&*{\bullet} \ar@{=}[rr]^{\{0,2\}}&&*{\bullet} \ar@{-}[rr]^1&&*{\bullet}&& \\
&& *{\bullet} \ar@{-}[rr]_0 &&*{\bullet} \ar@{-}[rr]_1&&*{\bullet}  \ar@{-}[rr]_0 &&*{\bullet} \ar@{-}[rr]_1&&*{\bullet} \ar@{-}[rr]_2&&*{\bullet}}$\qquad (2)\\

 &  $ \xymatrix@-1.5pc{  && && &&*{\bullet} \ar@{-}[d]_2 \ar@{-}[rr]^0&&*{\bullet}  \ar@{-}[rr]^1   \ar@{-}[d]_2&&*{\bullet} \ar@{=}[rr]^{\{0,2\}}&&*{\bullet} \ar@{-}[rr]^1&&*{\bullet} \ar@{-}[rr]^2&&*{\bullet}\\
*{\bullet} \ar@{-}[rr]_1&&*{\bullet}  \ar@{-}[rr]_0 &&*{\bullet} \ar@{-}[rr]_1&&*{\bullet} \ar@{-}[rr]_0&&*{\bullet}&&&&&&&& }$\qquad (3)\\
\hline
$\{6,6\}$& $ \xymatrix@-1.5pc{  &&*{\bullet} \ar@{-}[d]_2 \ar@{-}[rr]^0&&*{\bullet}  \ar@{-}[rr]^1   \ar@{-}[d]_2&&*{\bullet} \ar@{-}[rr]^2&&*{\bullet} \ar@{-}[rr]^1&&*{\bullet}&& \\
&& *{\bullet} \ar@{-}[rr]_0 &&*{\bullet} \ar@{-}[rr]_1&&*{\bullet}  \ar@{-}[rr]_0 &&*{\bullet} \ar@{-}[rr]_1&&*{\bullet} \ar@{=}[rr]_{\{0,2\}}&&*{\bullet}}$\qquad (4)\\
\hline
$\{3,5,3\}$ & $ \xymatrix@-1.3pc{  *{\bullet} \ar@{-}[rr]^0 &&*{\bullet} \ar@{-}[rr]^1&&*{\bullet}  \ar@{-}[rr]^2 &&*{\bullet}\ar@{-}[d]_3  \ar@{-}[rr]^1&&*{\bullet} \ar@{=}[rr]^{\{0,2\}}\ar@{-}[d]_3&&*{\bullet}\ar@{-}[d]^3\\
 &&&& &&*{\bullet} \ar@{-}[rr]_1&&*{\bullet} \ar@{-}[rr]^0\ar@{-}[d]^2&&*{\bullet}\ar@{-}[d]^2\\
 &&&&&&&&*{\bullet}  \ar@{-}[rr]\ar@{=}[rr]_{\{0,1,3\}}&&*{\bullet}}$ \qquad (5)\\
\hline
\end{tabular}\]
\end{prop}
\begin{proof}
As $\PSL_2(11)$ is simple and has exactly two distinct conjugacy classes of subgroups of index $11$ (which are isomorphic to $A_5$), there are exactly two faithful transitive permutation representation of $\PSL_2(11)$ on $11$ points.
The graphs (1) and (4) are the graphs for the regular polyhedra of types $\{5,5\}$ and $\{6,6\}$ respectively \cite{Hatlas}.
The other permutation representation graphs of these polyhedra are  obtained by interchanging the labels $0$ and $2$, that is, the dual graph of the ones presented.

The regular polyhedron of type $\{5,6\}$  has the permutation representations  graphs (2) and (3).
 As only (2) can be found in  \cite{Hatlas}, we obtained the graph (3) using the Todd-Coxeter Algorithm \cite{J}.
 
The graph (5) is the permutation representation graph of the unique rank $4$ string C-group by Proposition~\ref{isplit} and Corollary~\ref{casefrac}.
Similarly to what happens with the polyhedron of type $\{5,5\}$ and $\{6,6\}$, the other faithful permutation representation is the dual of (5).
\end{proof}


At this point, it remains to consider the cases where $G$ is either $A_{11}$ or $M_{11}$. From Corollary~\ref{casefrac}, $G_i$ must be transitive for some $i$. Let us consider the case where $G_i$ is $\PSL_2(11)$.

\begin{lemma}~\label{noPSL}
Let $r\in\{4,5\}$. If $G$ is any transitive even group of degree 11 of rank $r$, then $G_i$ cannot be $\PSL_2(11)$.
\end{lemma}
\begin{proof}
By the commuting property   $G_i=G_{<i}\times G_{>i}$  for $i\in\{1,\ldots, r-2\}$. 
Thus only $G_0$ or  $G_{r-1}$ can be simple groups.
 For a contradiction, and up to duality, let us assume that $G_0\cong \PSL_2(11)$. Then we find ten possibilites for $\mathcal{G}_0$ corresponding to  the five permutation representation graphs of Proposition~\ref{PSL} and their duals.
 We now use Lemmas~\ref{l1} and  \ref{l2} to reduce the possibilities due to the sizes of the orbits of $G_{0,1}$.
In the following table we list the possibilities for the sizes of the connected components of $\mathcal{G}_{0,1}$. 
For each of the graphs of Proposition~\ref{PSL} we need also to consider the duals, for this reason each cases gives two possibilities.
 \begin{center}
 \begin{tabular}{cc|cc}
 &Sizes of the &Dual of&Sizes of the \\
 Graph& $\mathcal{G}_{0,1}$-components&Graph & $\mathcal{G}_{0,1}$-components\\
 \hline
 (1) & $1, 5, 5$&(1) &$1,5,5$\\
 (2) &$2, 3, 6$&(2) &$1,5,5$\\
 (3) &$2,3,6$ &(3) &$1,5,5$\\
 (4) &$2,3,6$ &(4) &$2,3,6$\\
 (5) & $1,10$&(5) &$5,6$\\
 \end{tabular}
 \end{center}

 The cases where the sizes of the orbits are $1,5,5$  are excluded by Lemma~\ref{l1}. The case where the orbits have sizes $1,10$ or $5,6$ are excluded by Lemma~\ref{l2} and by the impossibility of having an even $\rho_0$.
 The remaining cases are those where the $G_1$-orbits have sizes $2, 3,6$. In these cases, by Lemma~\ref{l2}, $\rho_0$ must be fixed-point-free on the orbits of even size. The graph (2) and (3) give the permutation representation graphs (A) and (B) below, respectively. Taking the graph (4), we get, up to duality, the permutation representation graph (C). 
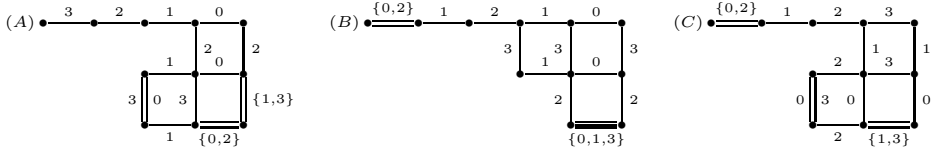
  \begin{tiny}
 \[\begin{array}{ccc}
{\mbox (A) }\xymatrix@-1.7pc{  *{\bullet} \ar@{-}[rr]^3 &&*{\bullet} \ar@{-}[rr]^2 &&*{\bullet}  \ar@{-}[rr]^1 &&*{\bullet}\ar@{-}[dd]^2  \ar@{-}[rr]^0 &&*{\bullet}\ar@{-}[dd]^2 \\
&&&&&&\\
 && &&*{\bullet} \ar@{-}[rr]^1 \ar@{=}[dd]_3^0 &&*{\bullet} \ar@{-}[rr]^0\ar@{-}[dd]_3&&*{\bullet}\ar@{=}[dd]^{\{1,3\}}\\
 &&&&&&\\
 &&&&*{\bullet}\ar@{-}[rr]_1 &&*{\bullet} \ar@{=}[rr]_{\{0,2\}}&&*{\bullet}}
 
 &
{ \mbox (B) }\xymatrix@-1.7pc{  *{\bullet} \ar@{=}[rr]^{\{0,2\}} &&*{\bullet} \ar@{-}[rr]^1&&*{\bullet}  \ar@{-}[rr]^2 &&*{\bullet}\ar@{-}[dd]_3  \ar@{-}[rr]^1&&*{\bullet} \ar@{-}[rr]^0\ar@{-}[dd]_3&&*{\bullet}\ar@{-}[dd]^3\\
&&&&&&&&&&\\
 &&&& &&*{\bullet} \ar@{-}[rr]^1&&*{\bullet} \ar@{-}[rr]^0\ar@{-}[dd]_2&&*{\bullet}\ar@{-}[dd]^2\\
 &&&&&&&&&&\\
 &&&&&&&&*{\bullet}  \ar@{-}[rr]\ar@{=}[rr]_{\{0,1,3\}}&&*{\bullet}}
 &
{\mbox (C) }\xymatrix@-1.7pc{  *{\bullet} \ar@{=}[rr]^{\{0,2\}} &&*{\bullet} \ar@{-}[rr]^1&&*{\bullet}  \ar@{-}[rr]^2 &&*{\bullet}\ar@{-}[dd]^1  \ar@{-}[rr]^3 &&*{\bullet}\ar@{-}[dd]^1 \\
&&&&&&&&\\
 && &&*{\bullet} \ar@{-}[rr]^2 \ar@{=}[dd]_0^3 &&*{\bullet} \ar@{-}[rr]^3\ar@{-}[dd]_0&&*{\bullet}\ar@{-}[dd]^0\\
 &&&&&&&&\\
 &&&&*{\bullet}\ar@{-}[rr]_2 &&*{\bullet} \ar@{=}[rr]_{\{1,3\}}&&*{\bullet}}
 
 \end{array}\]
  \end{tiny}
 In any case the $\rho_0 = (\rho_3\rho_2)^3 \in G_0$, contradicting the intersection property.
  \end{proof}

\begin{coro}\label{4,5M11}
There are no abstract regular polytopes of rank $4$ and $5$ for $M_{11}$.
\end{coro}

The non-existance of  abstract regular polyhedra for $M_{11}$ is a consequence of the following result.

\begin{lemma}\label{M}\cite{Mazurov}
None of the groups $G\in\{ M_{11}, M_{22}, M_{23}, M^cL \}$ has a generating set of three involutions two of which commute.
\end{lemma}

To summarize, we have proved the following result.\\

\begin{thm}~\label{main}
There is exactly one abstract regular polytope of rank $r\in\{4,5\}$ for an even transitive group of degree $11$, namely the $11$-cell, which is self-dual and of rank $4$.
\end{thm}
\begin{proof}
This is a consequence of Propositions~\ref{casefrac}, ~\ref{noPSL} and  \ref{4,5M11}.
\end{proof}


\section{Acknowledgements}

The author Maria Elisa Fernandes was supported by  the Center for Research and Development 
in Mathematics and Applications (CIDMA) through the Portuguese 
Foundation for Science and Technology 
(FCT - Fundação para a Ciência e a Tecnologia), 
references UIDB/04106/2020 and UIDP/04106/2020. The author Olivia Reade is grateful to the Open University for the travel grant enabling her to attend the 2022  SIGMAP conference which resulted in her involvement with this project. The author Claudio Alexandre Piedade was partially supported by CMUP, member of LASI, which is financed by national funds through FCT – Fundação para a Ciência e a Tecnologia, I.P., under the projects with reference UIDB/00144/2020 and UIDP/00144/2020.

 \bibliographystyle{amsplain}

\section{Appendix: Table of sggi's for $A_{11}$}

\begin{tiny}
\[\begin{array}{|cc|}
\hline
\begin{tabular}{cc}
 (A1) & \xymatrix@-1.9pc{*{\bullet} \ar@{-}[rr]^1&& *{\bullet}\ar@{-}[rr]^0&&*{\bullet}\ar@{-}[rr]^1&& *{\bullet}\ar@{=}[rr]^{\{0,2\}}&& *{\bullet}\ar@{-}[rr]^1&& *{\bullet} \ar@{-}[rr]^2&&  *{\bullet} \ar@{-}[rr]^3&& *{\bullet}\ar@{-}[rr]^2&&*{\bullet}\ar@{=}[rr]^{\{1,3\}}&&*{\bullet}\ar@{-}[rr]^2&&*{\bullet}} \\
 (A2) & \xymatrix@-1.9pc{*{\bullet} \ar@{-}[rr]^1&& *{\bullet}\ar@{-}[rr]^0&&*{\bullet}\ar@{-}[rr]^1&& *{\bullet}\ar@{=}[rr]^{\{0,2\}}&& *{\bullet}\ar@{-}[rr]^1&& *{\bullet} \ar@{-}[rr]^2&&  *{\bullet} \ar@{=}[rr]^{\{1,3\}}&& *{\bullet}\ar@{-}[rr]^2&&*{\bullet}\ar@{-}[rr]^3&&*{\bullet}\ar@{-}[rr]^2&&*{\bullet}}\\
 \end{tabular}
 &
 \begin{tabular}{cc}
 (A3) & \xymatrix@-1.9pc{*{\bullet} \ar@{-}[rr]^1&& *{\bullet}\ar@{-}[rr]^0&&*{\bullet}\ar@{-}[rr]^1&& *{\bullet}\ar@{-}[rr]^2&& *{\bullet}\ar@{=}[rr]^{\{1,3\}}&& *{\bullet} \ar@{-}[rr]^2&&  *{\bullet} \ar@{-}[rr]^3&& *{\bullet}\ar@{-}[rr]^2&&*{\bullet}\ar@{-}[rr]^1&&*{\bullet}\ar@{=}[rr]^{\{0,2\}}&&*{\bullet}}\\
 (A4) & \xymatrix@-1.9pc{*{\bullet} \ar@{-}[rr]^1&& *{\bullet}\ar@{-}[rr]^0&&*{\bullet}\ar@{-}[rr]^1&& *{\bullet}\ar@{-}[rr]^2&& *{\bullet}\ar@{-}[rr]^3&& *{\bullet} \ar@{-}[rr]^2&&  *{\bullet}\ar@{=}[rr]^{\{1,3\}}&& *{\bullet}\ar@{-}[rr]^2&&*{\bullet}\ar@{-}[rr]^1&&*{\bullet}\ar@{=}[rr]^{\{0,2\}}&&*{\bullet}}\\
 \end{tabular}\\[13pt]
\hline
\begin{tabular}{cc}
 (B1) & \xymatrix@-1.9pc{*{\bullet} \ar@{=}[rr]^{\{0,2\}}&& *{\bullet}\ar@{-}[rr]^1&&*{\bullet}\ar@{-}[rr]^0&& *{\bullet}\ar@{-}[rr]^1&&*{\bullet}\ar@{-}[rr]^2&& *{\bullet}\ar@{-}[rr]^3&&*{\bullet}\ar@{-}[rr]^2\ar@{-}[d]^4&&*{\bullet}\ar@{-}[rr]^3\ar@{-}[d]^4&&*{\bullet}\\
 && && && && && &&*{\bullet} \ar@{-}[rr]_2&&*{\bullet} &&}\\
 (B2) & \xymatrix@-1.9pc{*{\bullet} \ar@{=}[rr]^{\{0,2\}} && *{\bullet}\ar@{-}[rr]^1&&*{\bullet}\ar@{-}[rr]^0&& *{\bullet}\ar@{-}[rr]^1&&*{\bullet}\ar@{-}[rr]^2&& *{\bullet}\ar@{-}[rr]^3&&*{\bullet}\ar@{-}[rr]^2\ar@{-}[d]^4&&*{\bullet}\ar@{-}[d]^4
 &&\\
 && && && && && &&*{\bullet} \ar@{-}[rr]_2&&*{\bullet} \ar@{-}[rr]_3&&*{\bullet} }\\
 (B3) & \xymatrix@-1.9pc{*{\bullet} \ar@{=}[rr]^{\{0,2\}} && *{\bullet}\ar@{-}[rr]^1&&*{\bullet}\ar@{-}[rr]^0&& *{\bullet}\ar@{-}[rr]^1&&*{\bullet}\ar@{-}[rr]^2&& *{\bullet}\ar@{-}[rr]^3&&*{\bullet}\ar@{-}[rr]^2\ar@{-}[d]^4&&*{\bullet}\ar@{-}[d]^4\\
 && && &&&& && *{\bullet} \ar@{-}[rr]_3 &&*{\bullet}\ar@{-}[rr]_2&&*{\bullet} }\\
(B4)&  \xymatrix@-1.9pc{*{\bullet} \ar@{=}[rr]^{\{0,2\}}&& *{\bullet}\ar@{-}[rr]^1&&*{\bullet}\ar@{-}[rr]^0&& *{\bullet}\ar@{-}[rr]^1&&*{\bullet}\ar@{-}[rr]^2&& *{\bullet}\ar@{-}[rr]^1&& *{\bullet}\ar@{-}[rr]^2&&*{\bullet}\ar@{=}[rr]^{\{1,3\}}&&*{\bullet}\ar@{-}[rr]^2&& *{\bullet}\ar@{-}[rr]^3&&*{\bullet}}\\
(B5)&  \xymatrix@-1.9pc{*{\bullet} \ar@{=}[rr]^{\{0,2\}} && *{\bullet}\ar@{-}[rr]^1&&*{\bullet}\ar@{-}[rr]^0&& *{\bullet}\ar@{-}[rr]^1&&*{\bullet}\ar@{-}[rr]^2&& *{\bullet}\ar@{-}[rr]^1&& *{\bullet}\ar@{-}[rr]^2&&*{\bullet}\ar@{-}[rr]^3&&*{\bullet}\ar@{-}[rr]^2&& *{\bullet}\ar@{=}[rr]^{\{1,3\}}&&*{\bullet}} \\
 (B6)&  \xymatrix@-1.9pc{*{\bullet} \ar@{-}[rr]^2 && *{\bullet}\ar@{-}[rr]^1&&*{\bullet}\ar@{-}[rr]^0&& *{\bullet}\ar@{-}[rr]^1&&*{\bullet}\ar@{=}[rr]^{\{0,2\}}&& *{\bullet}\ar@{-}[rr]^1&& *{\bullet}\ar@{-}[rr]^2&&*{\bullet}\ar@{=}[rr]^{\{1,3\}}&&*{\bullet}\ar@{-}[rr]^2&& *{\bullet}\ar@{-}[rr]^3&&*{\bullet}} \\
(B7)&  \xymatrix@-1.9pc{*{\bullet} \ar@{-}[rr]^2 && *{\bullet}\ar@{-}[rr]^1&&*{\bullet}\ar@{-}[rr]^0&& *{\bullet}\ar@{-}[rr]^1&&*{\bullet}\ar@{=}[rr]^{\{0,2\}}&& *{\bullet}\ar@{-}[rr]^1&& *{\bullet}\ar@{-}[rr]^2&&*{\bullet}\ar@{-}[rr]^3&&*{\bullet}\ar@{-}[rr]^2&& *{\bullet}\ar@{=}[rr]^{\{1,3\}}&&*{\bullet}}  \\
(B8) & \xymatrix@-1.9pc{*{\bullet} \ar@{=}[rr]^{\{0,2\}} && *{\bullet}\ar@{-}[rr]^1&&*{\bullet}\ar@{-}[rr]^0&& *{\bullet}\ar@{-}[rr]^1&&*{\bullet}\ar@{-}[rr]^2&& *{\bullet}\ar@{-}[rr]^3\ar@{-}[d]_1&&*{\bullet}\ar@{-}[rr]^2\ar@{-}[d]^1&&*{\bullet}\\
 && && && && *{\bullet} \ar@{-}[rr]_2 &&*{\bullet} \ar@{-}[rr]_3&&*{\bullet} }\\
 \end{tabular}
 &
\begin{tabular}{cc}
 (B9) & \xymatrix@-1.9pc{*{\bullet} \ar@{=}[rr]^{\{0,2\}} && *{\bullet}\ar@{-}[rr]^1&&*{\bullet}\ar@{-}[rr]^0&& *{\bullet}\ar@{-}[rr]^1&&*{\bullet}\ar@{-}[rr]^2&& *{\bullet}\ar@{-}[rr]^3\ar@{-}[d]_1&&*{\bullet}\ar@{-}[d]^1&&\\
 && && && && *{\bullet} \ar@{-}[rr]_2 &&*{\bullet} \ar@{-}[rr]_3&&*{\bullet}\ar@{-}[rr]_2&&*{\bullet}  }\\
(B10) & \xymatrix@-1.9pc{*{\bullet} \ar@{=}[rr]^{\{0,2\}} && *{\bullet}\ar@{-}[rr]^1&&*{\bullet}\ar@{-}[rr]^0&& *{\bullet}\ar@{-}[rr]^1&&*{\bullet}\ar@{-}[rr]^2&& *{\bullet}\ar@{-}[rr]^3\ar@{-}[d]_1&&*{\bullet}\ar@{-}[rr]^2\ar@{-}[d]^1&&*{\bullet}\\
 && && && &&  &&*{\bullet} \ar@{-}[rr]_3&&*{\bullet}\ar@{-}[rr]_2 && *{\bullet}} \\
 (B11)&  \xymatrix@-1.9pc{*{\bullet} \ar@{=}[rr]^{\{0,2\}} && *{\bullet}\ar@{-}[rr]^1&&*{\bullet}\ar@{-}[rr]^0&& *{\bullet}\ar@{-}[rr]^1&&*{\bullet}\ar@{-}[rr]^2&& *{\bullet}\ar@{=}[rr]^{\{1,3\}}&& *{\bullet}\ar@{-}[rr]^2&&*{\bullet}\ar@{-}[rr]^1&&*{\bullet}\ar@{-}[rr]^2&& *{\bullet}\ar@{-}[rr]^3&&*{\bullet}}  \\
 (B12)&  \xymatrix@-1.9pc{*{\bullet} \ar@{=}[rr]^{\{0,2\}} && *{\bullet}\ar@{-}[rr]^1&&*{\bullet}\ar@{-}[rr]^0&& *{\bullet}\ar@{-}[rr]^1&&*{\bullet}\ar@{-}[rr]^2&& *{\bullet}\ar@{=}[rr]^{\{1,3\}} && *{\bullet}\ar@{-}[rr]^2&&*{\bullet}\ar@{-}[rr]^3 &&*{\bullet}\ar@{-}[rr]^2&& *{\bullet}\ar@{-}[rr]^1 &&*{\bullet}}  \\
 (B13) & \xymatrix@-1.9pc{*{\bullet} \ar@{=}[rr]^{\{0,2\}} && *{\bullet}\ar@{-}[rr]^1&&*{\bullet}\ar@{-}[rr]^0&& *{\bullet}\ar@{-}[rr]^1&&*{\bullet}\ar@{-}[rr]^0 \ar@{-}[d]_2 && *{\bullet}\ar@{-}[rr]^1\ar@{-}[d]^2&&*{\bullet}\ar@{-}[rr]^2 && *{\bullet}\ar@{-}[rr]^3 &&*{\bullet}\\
 && && && &&  *{\bullet} \ar@3{-}[rr]_{\{0,1,3\}}&&*{\bullet}} \\
 (B14)&  \xymatrix@-1.9pc{*{\bullet} \ar@{=}[rr]^{\{0,2\}} && *{\bullet}\ar@{-}[rr]^1&&*{\bullet}\ar@{-}[rr]^0&& *{\bullet}\ar@{-}[rr]^1&&*{\bullet}\ar@{-}[rr]^2&& *{\bullet}\ar@{-}[rr]^3&& *{\bullet}\ar@{-}[rr]^2&&*{\bullet}\ar@{=}[rr]^{\{1,3\}}&&*{\bullet}\ar@{-}[rr]^2&& *{\bullet}\ar@{-}[rr]^1&&*{\bullet}}\\
(B15)&  \xymatrix@-1.9pc{*{\bullet} \ar@{=}[rr]^{\{0,2\}} && *{\bullet}\ar@{-}[rr]^1&&*{\bullet}\ar@{-}[rr]^0&& *{\bullet}\ar@{-}[rr]^1&&*{\bullet}\ar@{-}[rr]^2&& *{\bullet}\ar@{-}[rr]^3&& *{\bullet}\ar@{-}[rr]^2&&*{\bullet}\ar@{-}[rr]^1&&*{\bullet}\ar@{-}[rr]^2&& *{\bullet}\ar@{=}[rr]^{\{1,3\}}&&*{\bullet}} \\
\end{tabular}\\[8pt]
\hline
\begin{tabular}{cc}
(C1)&  \xymatrix@-1.9pc{*{\bullet} \ar@{-}[rr]^1 && *{\bullet}\ar@{=}[rr]^{\{0,2\}}&&*{\bullet}\ar@{-}[rr]^1&& *{\bullet}\ar@{-}[rr]^0&&*{\bullet}\ar@{-}[rr]^1&& *{\bullet}\ar@{-}[rr]^2&& *{\bullet}\ar@{-}[rr]^3&&*{\bullet}\ar@{-}[rr]^2&&*{\bullet}\ar@{=}[rr]^{\{1,3\}}&& *{\bullet}\ar@{-}[rr]^2&&*{\bullet}}\\
(C2)&  \xymatrix@-1.9pc{*{\bullet} \ar@{-}[rr]^3 && *{\bullet}\ar@{-}[rr]^2&&*{\bullet}\ar@{-}[rr]^1&& *{\bullet}\ar@{-}[rr]^0&&*{\bullet}\ar@{-}[rr]^1&& *{\bullet}\ar@{-}[rr]^2&& *{\bullet}\ar@{=}[rr]^{\{1,3\}}&&*{\bullet}\ar@{-}[rr]^2&&*{\bullet}\ar@{-}[rr]^1&& *{\bullet}\ar@{=}[rr]^{\{0,2\}}&&*{\bullet}} \\
(C3)&  \xymatrix@-1.9pc{*{\bullet} \ar@{=}[rr]^{\{1,3\}} && *{\bullet}\ar@{-}[rr]^2&&*{\bullet}\ar@{-}[rr]^1&& *{\bullet}\ar@{-}[rr]^0&&*{\bullet}\ar@{-}[rr]^1&& *{\bullet}\ar@{-}[rr]^2&& *{\bullet}\ar@{-}[rr]^3&&*{\bullet}\ar@{-}[rr]^2&&*{\bullet}\ar@{-}[rr]^1&& *{\bullet}\ar@{=}[rr]^{\{0,2\}}&&*{\bullet}} 
\end{tabular}
&
\begin{tabular}{cc}
(C4)&  \xymatrix@-1.9pc{*{\bullet} \ar@{-}[rr]^3 && *{\bullet}\ar@{-}[rr]^2&&*{\bullet}\ar@{-}[rr]^1&& *{\bullet}\ar@{-}[rr]^0&&*{\bullet}\ar@{-}[rr]^1&& *{\bullet}\ar@{-}[rr]^2 \ar@{-}[d]^0&& *{\bullet}\ar@3{-}[d]^{\{0,1,3\}}\\
&&&&&&*{\bullet}\ar@{=}[rr]_{\{0,2\}}&&*{\bullet}\ar@{-}[rr]_1&&*{\bullet}\ar@{-}[rr]_2&& *{\bullet}\ } \\
(C5)&  \xymatrix@-1.9pc{*{\bullet} \ar@{-}[rr]^3 && *{\bullet}\ar@{-}[rr]^2&&*{\bullet}\ar@{-}[rr]^1&& *{\bullet}\ar@{-}[rr]^0&&*{\bullet}\ar@{-}[rr]^1&& *{\bullet}\ar@{=}[rr]^{\{0,2\}}&& *{\bullet}\ar@{-}[rr]^1&&*{\bullet}\ar@{-}[rr]^2&&*{\bullet}\ar@{=}[rr]^{\{1,3\}}&& *{\bullet}\ar@{-}[rr]^2&&*{\bullet}} \\
(C6)&  \xymatrix@-1.9pc{*{\bullet} \ar@{-}[rr]^3 && *{\bullet}\ar@{-}[rr]^2&&*{\bullet}\ar@{-}[rr]^1&& *{\bullet}\ar@{-}[rr]^0&&*{\bullet}\ar@{-}[rr]^1&& *{\bullet}\ar@{=}[rr]^{\{0,2\}} && *{\bullet}\ar@{-}[rr]^1 && *{\bullet}\ar@{-}[rr]^2  \ar@{-}[d]^0 && *{\bullet}\ar@3{-}[d]^{\{0,1,3\}}\\
&&&&&&&&&& && &&*{\bullet}\ar@{-}[rr]_2&& *{\bullet}} 
\end{tabular}\\[8pt]
\hline
\begin{tabular}{cc}
(D1)&  \xymatrix@-1.9pc{*{\bullet} \ar@{=}[rr]^{\{0,2\}} && *{\bullet}\ar@{-}[rr]^1&&*{\bullet}\ar@{-}[rr]^2&& *{\bullet}\ar@{-}[rr]^1&&*{\bullet}\ar@{-}[rr]^0&& *{\bullet}\ar@{-}[rr]^1&& *{\bullet}\ar@{-}[rr]^2&&*{\bullet}\ar@{=}[rr]^{\{1,3\}}&&*{\bullet}\ar@{-}[rr]^2&& *{\bullet}\ar@{-}[rr]^3&&*{\bullet}} \\
(D2)&  \xymatrix@-1.9pc{*{\bullet} \ar@{-}[rr]^2 && *{\bullet}\ar@{-}[rr]^1&&*{\bullet}\ar@{=}[rr]^{\{0,2\}}&& *{\bullet}\ar@{-}[rr]^1&&*{\bullet}\ar@{-}[rr]^0&& *{\bullet}\ar@{-}[rr]^1&& *{\bullet}\ar@{-}[rr]^2&&*{\bullet}\ar@{=}[rr]^{\{1,3\}}&&*{\bullet}\ar@{-}[rr]^2&& *{\bullet}\ar@{-}[rr]^3&&*{\bullet}} \\
\end{tabular}
&
\begin{tabular}{cc}
(D3)&  \xymatrix@-1.9pc{*{\bullet} \ar@{-}[rr]^2 && *{\bullet}\ar@{=}[rr]^{\{1,3\}} &&*{\bullet}\ar@{-}[rr]^2 && *{\bullet}\ar@{-}[rr]^1&&*{\bullet}\ar@{-}[rr]^0&& *{\bullet}\ar@{-}[rr]^1&& *{\bullet}\ar@{=}[rr]^{\{0,2\}}&&*{\bullet}\ar@{-}[rr]^1&&*{\bullet}\ar@{-}[rr]^2&& *{\bullet}\ar@{-}[rr]^3&&*{\bullet}} \\
(D4)&  \xymatrix@-1.9pc{*{\bullet}\ar@{-}[rr]^2 \ar@3{-}[d]_{\{0,1,3\}} && *{\bullet}\ar@{-}[rr]^1 \ar@{-}[d]^{0} &&*{\bullet}\ar@{-}[rr]^0&& *{\bullet}\ar@{-}[rr]^1&& *{\bullet}\ar@{=}[rr]^{\{0,2\}}&&*{\bullet}\ar@{-}[rr]^1&&*{\bullet}\ar@{-}[rr]^2&& *{\bullet}\ar@{-}[rr]^3&&*{\bullet}\\
*{\bullet}\ar@{-}[rr]_2 && *{\bullet}} 
\end{tabular}\\[8pt]
\hline
\begin{tabular}{cc}
   (E1) & \xymatrix@-1.9pc{ && && && && *{\bullet}\ar@3{-}[rr]^{\{0,1,3\}}&& *{\bullet}\\
   &&&&&&&&&&\\
 *{\bullet}\ar@{-}[rr]^0 && *{\bullet}\ar@{-}[rr]^1 && *{\bullet}\ar@{-}[rr]^2 &&*{\bullet}\ar@{-}[rr]^1 && *{\bullet}\ar@{-}[uu]_2 \ar@{-}[rr]_0 && *{\bullet}\ar@{-}[uu]_2\\
    &&&&&&&&&&\\
  && && && *{\bullet}\ar@{-}[rr]_1\ar@{-}[uu]^3 && *{\bullet}\ar@{-}[uu]^3\ar@{=}[rr]_{\{0,2\}} & & *{\bullet}\ar@{-}[uu]_3 \\
   } \\

(E2) & \xymatrix@-1.9pc{  *{\bullet}\ar@{-}[rr]^0 && *{\bullet}\ar@{-}[rr]^1 && *{\bullet}\ar@{-}[rr]^2 && *{\bullet}\ar@{-}[rr]^1 && *{\bullet}\\
  *{\bullet}\ar@{=}[rr]_{\{0,2\}} && *{\bullet}\ar@{-}[rr]_1 && *{\bullet}\ar@{-}[rr]_2 && *{\bullet}\ar@{-}[rr]_1 \ar@{-}[u]^3 && *{\bullet}\ar@{-}[rr]_2 \ar@{-}[u]_3 && *{\bullet}} \\
(E3) & \xymatrix@-1.9pc{  *{\bullet}\ar@{-}[rr]^0 && *{\bullet}\ar@{-}[rr]^1 && *{\bullet}\ar@{-}[rr]^2 && *{\bullet}\ar@{-}[rr]^1 && *{\bullet}\ar@{-}[rr]^2 && *{\bullet}\\
  *{\bullet}\ar@{=}[rr]_{\{0,2\}} && *{\bullet}\ar@{-}[rr]_1 && *{\bullet}\ar@{-}[rr]_2 && *{\bullet}\ar@{-}[rr]_1 \ar@{-}[u]^3 && *{\bullet} \ar@{-}[u]_3 } \\
  (E4) & \xymatrix@-1.9pc{  *{\bullet}\ar@{-}[rr]^0 && *{\bullet}\ar@{-}[rr]^1 && *{\bullet}\ar@{=}[rr]^{\{0,2\}} && *{\bullet}\ar@{-}[rr]^1 && *{\bullet}\ar@{-}[rr]^2 && *{\bullet}\ar@{-}[rr]^1 && *{\bullet}\ar@{-}[rr]^2 && *{\bullet}\\ 
  &&  && && && *{\bullet}\ar@{-}[rr]_2 && *{\bullet}\ar@{-}[rr]_1 \ar@{-}[u]^3 && *{\bullet} \ar@{-}[u]_3  }  \\ 
    (E5) & \xymatrix@-1.9pc{  *{\bullet}\ar@{-}[rr]^0 && *{\bullet}\ar@{-}[rr]^1 && *{\bullet}\ar@{=}[rr]^{\{0,2\}} && *{\bullet}\ar@{-}[rr]^1 && *{\bullet}\ar@{-}[rr]^2 && *{\bullet}\ar@{-}[rr]^1 && *{\bullet} \\
  &&  && && && *{\bullet}\ar@{-}[rr]_2 && *{\bullet}\ar@{-}[rr]_1 \ar@{-}[u]^3 && *{\bullet} \ar@{-}[u]_3  \ar@{-}[rr]_2 && *{\bullet}} \\  
    (E6) & \xymatrix@-1.9pc{  *{\bullet}\ar@{-}[rr]^0 && *{\bullet}\ar@{-}[rr]^1 && *{\bullet}\ar@{=}[rr]^{\{0,2\}} && *{\bullet}\ar@{-}[rr]^1 && *{\bullet}\ar@{-}[rr]^2 && *{\bullet}\ar@{-}[rr]^1 && *{\bullet} \ar@{-}[rr]^2 && *{\bullet}\\
  &&  && && && && *{\bullet}\ar@{-}[rr]_1 \ar@{-}[u]^3 && *{\bullet} \ar@{-}[u]_3  \ar@{-}[rr]_2 && *{\bullet} }  \\  
    \end{tabular}
  &
  \begin{tabular}{cc}
  (E7) & \xymatrix@-1.9pc{  *{\bullet}\ar@{-}[rr]^0 && *{\bullet}\ar@{-}[rr]^1 && *{\bullet}\ar@{-}[rr]^2 && *{\bullet}\ar@{-}[rr]^1 && *{\bullet}\ar@{-}[rr]^2 && *{\bullet}\ar@{-}[rr]^1 && *{\bullet}\ar@{=}[rr]_{\{0,2\}} && *{\bullet} \\
  &&  && && *{\bullet}\ar@{-}[rr]_1 \ar@{-}[u]^3 && *{\bullet} \ar@{-}[u]_3  \ar@{-}[rr]_2 && *{\bullet}} \\ 
  (E8) & \xymatrix@-1.9pc{  *{\bullet}\ar@{-}[rr]^0 && *{\bullet}\ar@{-}[rr]^1 && *{\bullet}\ar@{-}[rr]^2 && *{\bullet}\ar@{-}[rr]^1 && *{\bullet}\ar@{-}[rr]^2 && *{\bullet}\ar@{-}[rr]^1 && *{\bullet}\ar@{=}[rr]_{\{0,2\}} && *{\bullet} \\
  &&  &&  *{\bullet}\ar@{-}[rr]_2 && *{\bullet}\ar@{-}[rr]_1 \ar@{-}[u]^3 && *{\bullet} \ar@{-}[u]_3 } \\
  (E9) & \xymatrix@-1.9pc{  *{\bullet}\ar@{-}[rr]^0 && *{\bullet}\ar@{-}[rr]^1 && *{\bullet}\ar@{-}[rr]^2 && *{\bullet}\ar@{-}[rr]^1 && *{\bullet}\ar@{-}[rr]^2 && *{\bullet} \\  
  &&  && && *{\bullet}\ar@{-}[rr]_1 \ar@{-}[u]^3 && *{\bullet} \ar@{-}[u]_3  \ar@{-}[rr]_2 && *{\bullet} \ar@{-}[rr]_1 && *{\bullet} \ar@{=}[rr]_{\{0,2\}}  && *{\bullet}} \\
    (E10) & \xymatrix@-1.9pc{  *{\bullet}\ar@{-}[rr]^0 && *{\bullet}\ar@{-}[rr]^1 && *{\bullet}\ar@{-}[rr]^2 && *{\bullet}\ar@{-}[rr]^1 && *{\bullet} \\  
  &&  && *{\bullet}\ar@{-}[rr]_2 && *{\bullet}\ar@{-}[rr]_1 \ar@{-}[u]^3 && *{\bullet} \ar@{-}[u]_3  \ar@{-}[rr]_2 && *{\bullet} \ar@{-}[rr]_1 && *{\bullet} \ar@{=}[rr]_{\{0,2\}}  && *{\bullet}} \\
  (E11)  &  \xymatrix@-1.9pc{*{\bullet}\ar@{-}[rr]^0 && *{\bullet}\ar@{-}[rr]^1 && *{\bullet}\ar@{=}[rr]^{\{0,2\}} && *{\bullet}\ar@{-}[rr]^1 && *{\bullet}\ar@{-}[rr]^2 && *{\bullet}\ar@{-}[rr]^3 && *{\bullet}\ar@{-}[rr]^4 && *{\bullet}\ar@{-}[rr]^3 &&  *{\bullet}\\
  && && && && && && *{\bullet}\ar@{-}[rr]_4 \ar@{-}[u]^2  && *{\bullet}\ar@{-}[u]_2  } \\
  (E12)  & \xymatrix@-1.9pc{*{\bullet}\ar@{-}[rr]^0 && *{\bullet}\ar@{-}[rr]^1 && *{\bullet}\ar@{=}[rr]^{\{0,2\}} && *{\bullet}\ar@{-}[rr]^1 && *{\bullet}\ar@{-}[rr]^2 && *{\bullet}\ar@{-}[rr]^3 && *{\bullet}\ar@{-}[rr]^4 && *{\bullet}\\
  && && && && && *{\bullet}\ar@{-}[rr]_3 && *{\bullet}\ar@{-}[rr]_4 \ar@{-}[u]^2  && *{\bullet}\ar@{-}[u]_2  } \\
  (E13) &  \xymatrix@-1.9pc{*{\bullet}\ar@{-}[rr]^0 && *{\bullet}\ar@{-}[rr]^1 && *{\bullet}\ar@{=}[rr]^{\{0,2\}} && *{\bullet}\ar@{-}[rr]^1 && *{\bullet}\ar@{-}[rr]^2 && *{\bullet}\ar@{-}[rr]^3 && *{\bullet}\ar@{-}[rr]^4 && *{\bullet}\\
  && && && && && && *{\bullet}\ar@{-}[rr]_4 \ar@{-}[u]^2  && *{\bullet}\ar@{-}[u]_2 \ar@{-}[rr]_3 &&  *{\bullet}}  \\
  \end{tabular}\\[8pt]
 \hline
\begin{tabular}{cc}
(F1) &  \xymatrix@-1.9pc{  &&*{\bullet} \ar@{-}[rr]^2 \ar@{-}[d]^0 && *{\bullet}  \ar@{-}[d]^0 && &&  *{\bullet} \ar@{-}[rr]^1 \ar@{-}[d]^3 && *{\bullet} \ar@{-}[d]^3 && \\
 *{\bullet} \ar@{-}[rr]_1 && *{\bullet}\ar@{-}[rr]_2 && *{\bullet}\ar@{-}[rr]_1 && *{\bullet}\ar@{-}[rr]_2 && *{\bullet}\ar@{-}[rr]_1 && *{\bullet}\ar@{-}[rr]_2 && *{\bullet}}  \\

(F2) &  \xymatrix@-1.9pc{  &&*{\bullet} \ar@{-}[rr]^2 \ar@{-}[d]^0 && *{\bullet}  \ar@{-}[d]^0 && &&  *{\bullet} \ar@{-}[rr]^1 \ar@{-}[d]^3 && *{\bullet} \ar@{-}[d]^3 \ar@{-}[rr]^2 && *{\bullet}\\
 *{\bullet} \ar@{-}[rr]_1 && *{\bullet}\ar@{-}[rr]_2 && *{\bullet}\ar@{-}[rr]_1 && *{\bullet}\ar@{-}[rr]_2 && *{\bullet}\ar@{-}[rr]_1 && *{\bullet}&& }  \\

(F3) &  \xymatrix@-1.9pc{  &&*{\bullet} \ar@{-}[rr]^2 \ar@{-}[d]^0 && *{\bullet}  \ar@{-}[d]^0 && &&  *{\bullet} \ar@{-}[rr]^3 \ar@{-}[d]^1 && *{\bullet} \ar@{-}[d]^1 && \\
 *{\bullet} \ar@{-}[rr]_1 && *{\bullet}\ar@{-}[rr]_2 && *{\bullet}\ar@{-}[rr]_1 && *{\bullet}\ar@{-}[rr]_2 && *{\bullet}\ar@{-}[rr]_3 && *{\bullet}\ar@{-}[rr]_2 && *{\bullet}}  \\

(F4) &  \xymatrix@-1.9pc{  *{\bullet} \ar@{-}[rr]^1 &&*{\bullet} \ar@{-}[rr]^2 \ar@{-}[d]^0 && *{\bullet}  \ar@{-}[d]^0 && &&  *{\bullet} \ar@{-}[rr]^1 \ar@{-}[d]^3 && *{\bullet} \ar@{-}[d]^3 \ar@{-}[rr]^2 && *{\bullet}\\
 && *{\bullet}\ar@{-}[rr]_2 && *{\bullet}\ar@{-}[rr]_1 && *{\bullet}\ar@{-}[rr]_2 && *{\bullet}\ar@{-}[rr]_1 && *{\bullet}&& }  \\

(F5) &  \xymatrix@-1.9pc{  *{\bullet} \ar@{-}[rr]^1 &&*{\bullet} \ar@{-}[rr]^2 \ar@{-}[d]^0 && *{\bullet}  \ar@{-}[d]^0 && &&  *{\bullet} \ar@{-}[rr]^3 \ar@{-}[d]^1 && *{\bullet} \ar@{-}[d]^1 && \\
 && *{\bullet}\ar@{-}[rr]_2 && *{\bullet}\ar@{-}[rr]_1 && *{\bullet}\ar@{-}[rr]_2 && *{\bullet}\ar@{-}[rr]_3 && *{\bullet}\ar@{-}[rr]_2 && *{\bullet}} \\
\end{tabular}
&
\begin{tabular}{cc}
(F6) &  \xymatrix@-1.9pc{  &&*{\bullet} \ar@{-}[rr]^0 \ar@{-}[d]^2 && *{\bullet}  \ar@{-}[d]^2 && &&  *{\bullet} \ar@{-}[rr]^3 \ar@{-}[d]^1 && *{\bullet} \ar@{-}[d]^1 && \\
 *{\bullet} \ar@{-}[rr]_1 && *{\bullet}\ar@{-}[rr]_0 && *{\bullet}\ar@{-}[rr]_1 && *{\bullet}\ar@{-}[rr]_2 && *{\bullet}\ar@{-}[rr]_3 && *{\bullet}\ar@{-}[rr]_2 && *{\bullet}} \\
 
(F7) &  \xymatrix@-1.9pc{  *{\bullet} \ar@{-}[rr]^2 \ar@{-}[d]^0 && *{\bullet}  \ar@{-}[d]^0 && &&  *{\bullet} \ar@{-}[rr]^1 \ar@{-}[d]^3 && *{\bullet} \ar@{-}[d]^3 && && \\
 *{\bullet}\ar@{-}[rr]_2 && *{\bullet}\ar@{-}[rr]_1 && *{\bullet}\ar@{-}[rr]_2 && *{\bullet}\ar@{-}[rr]_1 && *{\bullet}\ar@{-}[rr]_2 && *{\bullet}\ar@{-}[rr]_1 && *{\bullet}}  \\

(F8) &  \xymatrix@-1.9pc{  *{\bullet} \ar@{-}[rr]^2 \ar@{-}[d]^0 && *{\bullet}  \ar@{-}[d]^0 && &&  *{\bullet} \ar@{-}[rr]^1 \ar@{-}[d]^3 && *{\bullet} \ar@{-}[d]^3 \ar@{-}[rr]^2 && *{\bullet} \ar@{-}[rr]_1 && *{\bullet}\\
 *{\bullet}\ar@{-}[rr]_2 && *{\bullet}\ar@{-}[rr]_1 && *{\bullet}\ar@{-}[rr]_2 && *{\bullet}\ar@{-}[rr]_1 && *{\bullet}&& && }  \\

(F9) & \xymatrix@-1.9pc{  *{\bullet} \ar@{-}[rr]^2 \ar@{-}[d]^0 && *{\bullet}  \ar@{-}[d]^0 && &&  *{\bullet} \ar@{-}[rr]^3 \ar@{-}[d]^1 && *{\bullet} \ar@{-}[d]^1 && && \\
 *{\bullet}\ar@{-}[rr]_2 && *{\bullet}\ar@{-}[rr]_1 && *{\bullet}\ar@{-}[rr]_2 && *{\bullet}\ar@{-}[rr]_3 && *{\bullet}\ar@{-}[rr]_2 && *{\bullet}\ar@{-}[rr]_1 && *{\bullet}}  \\

(F10) &  \xymatrix@-1.9pc{  *{\bullet} \ar@{-}[rr]^2 \ar@{-}[d]^0 && *{\bullet}  \ar@{-}[d]^0 && && && &&  *{\bullet} \ar@{-}[rr]^1 \ar@{-}[d]_3 && *{\bullet} \ar@{-}[d]_3 \\
 *{\bullet}\ar@{-}[rr]_2 && *{\bullet}\ar@{-}[rr]_1 && *{\bullet}\ar@{-}[rr]_2 && *{\bullet}\ar@{-}[rr]_1 && *{\bullet}\ar@{-}[rr]_2 && *{\bullet}\ar@{-}[rr]_1 && *{\bullet}}  \\ 
 \end{tabular}\\
\hline
\end{array} \]

\end{tiny}

\end{document}